\documentclass[journal,twoside,web]{IEEEtran}
\usepackage{generic}
\usepackage{amsmath,amssymb,amsfonts}
\usepackage{algorithmic}
\usepackage{xcolor}
\usepackage{graphicx}
\usepackage{algorithm,algorithmic}
\usepackage{mathtools}
\usepackage{modroman}
\usepackage{xspace}
\usepackage{bbm}
\usepackage{bold-extra}
\usepackage[most]{tcolorbox}
\usepackage{diagbox}
\usepackage{multirow}
\usepackage{textcomp}
\usepackage{braket}

\usepackage{array}

\usepackage[caption=false]{subfig}
\usepackage[flushleft]{threeparttable}

\usepackage{pgfplots}

\usepackage{algorithmic}

\usepackage[style=ieee]{biblatex}
\usepackage{amsopn}
\makeatletter
\newcommand\notsotiny{\@setfontsize\notsotiny{6}{7}}
\makeatother

\newtheorem{theorem}{Theorem}
\newtheorem{lemma}{Lemma}

\newtheorem{proposition}{Proposition}
\newtheorem{definition}{Definition}
\newtheorem{assumption}{Assumption}

\newtheorem{remark}{Remark}

\allowdisplaybreaks
\begin{document}
\title{The landscape of deterministic and stochastic optimal control problems: One-shot Optimization versus Dynamic Programming}
\author{Jihun Kim, Yuhao Ding, Yingjie Bi, and Javad Lavaei, \textit{Fellow, IEEE}
\thanks{This work was supported by grants from ARO, ONR, AFOSR, NSF, and the UC Noyce Initiative.}
\thanks{J. Kim, Y. Ding, Y. Bi, and J. Lavaei are with the Department of Industrial Engineering and Operations Research, University of California, Berkeley, CA 94720 USA (e-mail: jihun.kim@berkeley.edu; yuhao\_ding@berkeley.edu; yb236@cornell.edu; lavaei@berkeley.edu).}
\thanks{A preliminary version of this paper has appeared in 2021 American Control Conference, New Orleans, USA, May 25-28, 2021 \cite{ding2021analysis}. The previous version mainly discussed the deterministic problem with control inputs, while this journal version has significantly extended the results to include both the deterministic and the stochastic problems under a parameterized policy to study a closed-loop system. To address the parameterized problem, our new notion of a local minimizer of the one-shot optimization optimizes the objective function over the parameters modeling the inputs. Furthermore, the notion of a locally minimum control policy of DP is replaced with a local minimizer of DP, which considers a neighborhood in the parameter space instead of the action space. These new notions enable the investigation of the stochastic dynamics as a finite-dimensional problem.}}

\maketitle

\begin{abstract}
% Dynamic programming (DP) has a rich theoretical foundation and a broad range of applications, especially in the classic area of optimal control and the newer area of reinforcement learning (RL). 
% Many optimal control problems can be solved via a single optimization problem, named one-shot optimization, or via a sequence of optimization problems using DP.
Optimal control problems can be solved via a one-shot (single) optimization or a sequence of optimization using dynamic programming (DP).
However, the computation of their global optima often faces NP-hardness, and thus only locally optimal solutions may be obtained at best. In this work, we consider the discrete-time finite-horizon optimal control problem in both deterministic and stochastic cases and study the optimization landscapes associated with two different approaches: one-shot and DP. In the deterministic case, we prove that each local minimizer of the one-shot optimization corresponds to some control input induced by a locally minimum control policy of DP, and vice versa. 
However, with a parameterized policy approach, we prove that deterministic and stochastic cases both exhibit the desirable property that each local minimizer of DP corresponds to some local minimizer of the one-shot optimization, but the converse does not necessarily hold. Nonetheless, under different technical assumptions for deterministic and stochastic cases, if there exists only a single locally minimum control policy, one-shot and DP turn out to capture the same local solution. 
These results pave the way to understand the performance and stability of local search methods in optimal control.
% and reinforcement learning.
% We further incorporate the parameterized policy approach into our problem for the closed-loop system. 
\end{abstract}


\section{Introduction}
\label{sec:intro}

Dynamic Programming (DP) has been widely used in a variety of fields with a rich theoretical foundation and many mathematical and algorithmic aspects \cite{bellman1957, bertsekas1995dynamic}. One classic area of DP is to solve optimal control problems, with applications in communication systems \cite{bellman1957role}, inventory control \cite{feinberg2016optimality}, powertrain control \cite{kolmanovsky2002optimization}, and many more. Furthermore, many recent successes in artificial intelligence, especially in reinforcement learning (RL) \cite{bertsekas2019reinforcement, sutton2018reinforcement}, are also deeply rooted in DP. In the challenging domain of classic Atari 2600 games, the work \cite{mnih2015human} has demonstrated that the Q-learning method based on the generalized policy iteration along with a deep neural network as the function approximator for the Q-values surpasses the performance of all previous algorithms and achieves a level comparable to that of a professional human games tester.  

Despite a strong theoretical framework of DP, the exact solutions of large-scale optimal control problems are often impossible to obtain using DP in practice \cite{bertsekas2019reinforcement}. Apart from suffering the ``curse of dimensionality'' when the state space is large, solving DP accurately could also be highly complex. 
The reason is that DP requires solving optimization sub-problems to global optimality, and the computation of their global optima is NP-hard in general, due to the non-linearity of the dynamics and the non-convexity of the cost function.

Therefore, although DP theory relies on global optimization solvers, practitioners routinely use local optimization solvers based on first- and second-order numerical algorithms. As a result, the theoretical guarantee of DP could break down as soon as a non-global local solution is found in any of the sub-problems. Understanding the performance of local search methods for non-convex problems has been a focal area in machine learning in recent years. This is performed under the notion of spurious solution, which refers to a local minimum that is not a global solution. The specific application areas are neural networks \cite{soltanolkotabi2019theoretical, liu2022landscape}, deep learning \cite{allen2019convergence, liang2022revisit}, mixtures of regressions \cite{mei2018landscape, chandrasekher2023sharp}, matrix sensing/recovery \cite{molybogrole, zhang2021sharp, zhang2022improved, ma2022sharp}, phase retrieval \cite{chen2019gradient, chandrasekher2023sharp}, and online optimization \cite{ding2023escape, fattahi2023absence}.
% in this case only the stationary points or the locally optimal points can be expected if we apply the local search methods like the gradient descent, Newton's method or their variants. Then a natural question is what is the quality of the solution obtained by replacing the global optimizers needed for solving each sub-optimization problems in DP with the local optimizers. 

Recently, there has been an increasing interest in understanding the global convergence of exact or approximate DP algorithms in policy gradient methods for RL, such as projected policy gradient, natural policy gradient, and mirror descent with or without regularizers \cite{agarwal2021policygradient, bhandari2021policygradient, xiao2022policygradient, lan2023policy, zhan2023policy}. Prior to them, the work \cite{bhandari2024global} identified some general algorithm-independent properties of the policy gradient method by establishing a direct connection between policy gradient (one-shot) and policy iteration (DP) objectives. They showed that the global convergence of the policy gradient method is guaranteed if the policy iteration objectives have no sub-optimal stationary points. However, the literature lacks a rigorous analysis of the spurious solutions of the DP method.

In this paper, we analyze the spurious solutions of the DP method by focusing on the following fundamental question: \textit{What if the globally optimal solution of each sub-problem of DP is replaced with a solution obtained by a local search method?} A challenge in this analysis is that policy optimization even towards the spurious solutions can be problematic if the action space is continuous \cite{Silver2014DeterministicPG}. One can think of the policy iteration with function approximation \cite{sutton1999policygradient} where the Q-function approximation error is zero. This is a reasonable assumption since a close-to-zero error can be obtained with a sufficiently rich and expressive policy class such as deep neural networks, which naturally yields the existence of the local minimizer of DP. That motivates our analysis on the comparison between the solutions of one-shot method and DP if they are only solved to the spurious local minimizers, and hence, our algorithm-agnostic study offers a clear understanding on the landscapes for the optimal control problem without considering the secondary issue of the approximation error.

We focus on both deterministic and stochastic discrete-time finite-horizon optimal control problems whose goal is to find an optimal input sequence minimizing the total cost subject to the dynamics. One approach to solving the problem is by formulating it as a one-shot optimization problem, a single entire-period problem, and another approach is using the DP to formulate it as a sequential decision-making problem with multiple single-period sub-problems and solve it in a backward way. Although it is known that the one-shot method and the DP method return the same globally optimal control sequence for the deterministic optimal control problem \cite{bertsekas1995dynamic}, it is not yet known what would occur if the global optimizer needed for solving each sub-optimization problem in DP is replaced by a local optimizer. In our work, we compare the two optimization landscapes: one induced by the DP method based on local search algorithms, and the other induced by its corresponding one-shot optimization based on local search methods. 

\textbf{Contribution and Outline.} We address the relationship between the two landscapes holistically for three types of control systems:

1. In Section \ref{sec:det}, we first study \textit{deterministic} systems under a non-parameterized policy. We introduce the notion of locally minimum control policy of DP and prove that under some mild conditions, each (spurious) local minimizer of the one-shot optimization corresponds to the control input induced by a (spurious) locally minimum control policy of DP, and vice versa. This indicates that DP with local search can successfully solve the optimal control problem to global optimality if and only if the one-shot problem is free of spurious solutions.

2. In Section \ref{sec:detparam}, we analyze \textit{deterministic} systems under a \textit{parameterized} policy. The necessity to study this problem arises in RL algorithms, where the control policy used by DP is parameterized by neural networks or other means. Thus, we generalize the results of Section \ref{sec:det} to optimization with respect to the parameters rather than the control inputs themselves. We prove that each local minimizer of DP corresponds to some local minimizer of the one-shot optimization, whereas its converse may not hold. Moreover, we show that if there exists only a single locally minimum control policy with a specific parameterized policy class, namely a linear combination of independent basis functions, each local minimizer of the one-shot optimization corresponds to a local minimizer of DP. 

3. In Section \ref{sec:stoparam}, we extend the result to \textit{stochastic} systems under a \textit{parameterized} policy. The stochasticity brings up the challenge to handle an uncountable number of realizations of random variables. We show that surprisingly a similar relationship in the deterministic parameterized problem holds. For both cases, 
we conclude that the optimization landscape of the one-shot problem is more complex than its DP counterpart
in terms of the number of spurious solutions. 
This implies that if the one-shot problem has a \textit{low} complexity, so does the DP problem. Another result says that if the DP problem has a \textit{very low} complexity, the same holds for the one-shot problem. In this paper, our notion of ``complexity" of an optimization problem is based on \textit{the number of spurious local minima}. For example, convex optimization problems have very low complexity in light of having no spurious solutions. However, problems with an exponential number of spurious solutions are hard to solve \cite{yalcin2022exp}. Note that a reformulation of an optimization problem via a change of variables does not normally change the number of local minima, which justifies why the number of spurious solutions can serve as a complexity measure.

Finally, concluding remarks are provided in Section \ref{sec:con}. Table \ref{intro table} summarizes the main results of the paper.

\begin{table}[t]
\notsotiny
% \scriptsize
\centering
\caption{Theorems and the corresponding assumptions with results.}
\begin{tabular}{|l|l||*{9}{c|}}
\hline
\multicolumn{2}{|l||}
{\multirow{3}{*}{\backslashbox[30mm]{Assumptions}{Theorems}}}
&\multicolumn{3}{|c|}{\multirow{2}{*}{Deterministic}}
&\multicolumn{3}{|c|}{Deterministic +}
&\multicolumn{3}{|c|}{Stochastic +} \\ 
\multicolumn{2}{|l||}{} & \multicolumn{3}{|c|}{}
&\multicolumn{3}{|c|}{Parameterized}
&\multicolumn{3}{|c|}{Parameterized}
\\ \cline{3-11}
\multicolumn{2}{|l||}{}   & \ref{thm: DP to one-shot epsilon PD} & \ref{thm: stationary} & \ref{thm: one-shot to DP} & \ref{thm: detparam loc min} & \ref{thm: detparam stationary} & \ref{thm: detparam one-shot to DP} & \ref{thm: stoparam loc min} & \ref{thm: stoparam stationary} & \ref{thm: stoparam one-shot to DP}

\\\hline\hline
\multirow{2}{*}{Convex} & action space&$\circ$&&&&&&&& \\ \cline{2-11} & parameter space&&&&&&$\circ$&&& \\  \hline
 & $C^1$ &&$\circ$&&&$\circ$&&&$\circ$& \\ 
\cline{2-11} Policy & $C^2$ &$\circ$&&&&&&&& \\
\cline{2-11} class & Defined by &&&&&& \multirow{2}{*}{$\circ$} &&& \multirow{2}{*}{$\circ$} \\
&Definition \ref{def: linear combination} &&&&&&&&& \\
\cline{2-11} & contains a single  &&&&&&&&& \\
& locally minimum  &&&&&&$\circ$&&&$\circ$ \\
& control policy  &&&&&&&&& \\\hline

\multicolumn{2}{|l||}{Interior policy} &&$\circ$&&&&&&&$\circ$ \\ \hline
% \multicolumn{2}{|l||}{Overparameterized policy} &&&&&&&$\circ$&&& \\ \hline
\multicolumn{2}{|l||}{Strict local minimizer} &&&$\circ$&&&$\circ$&&& \\ \hline
\multicolumn{2}{|l||}{Continuous Random state} &&&&&&&&&$\circ$ \\ \hline
\multicolumn{2}{|l||}{Large parameter space} &&&&&&$\circ$&&& \\ \hline\hline
\multirow{4}{*}{Result} & DP to one-shot &$\circ$&&&$\circ$&&&$\circ$&& \\ \cline{2-11}
& DP to one-shot &&\multirow{2}{*}{$\circ$}&&&\multirow{2}{*}{$\circ$}&&&\multirow{2}{*}{$\circ$}& \\ 
& (stationarity) &&&&&&&&& \\ \cline{2-11}
& one-shot to DP &&&$\circ$&&&$\circ$&&& $\circ$ \\ \hline
\end{tabular}
\vspace{-3.9mm}
\label{intro table}
\end{table}

In various applications arising in machine learning and model-free approaches for which the model is unknown and simulations are expensive, DP is the only viable choice compared to the one-shot optimization approach. Hence, it is essential to understand when DP combined with a local search solver works. The results of this paper explain that the success of DP is closely related to the optimization landscape of a single optimization problem. For instance, the success of the DP method highly depends on the number of spurious solutions of the one-shot optimization problem.

% The goal of this paper is to show that, under mild conditions, this method works if and only if its corresponding one-shot optimization can be successfully solvable using local search methods. 

% The remainder of this paper is organized as follows. Problem formulations are given in Section \ref{sec: prob formulation}. The main results on the relationship between the (spurious) local minimizer of the one-shot problem and the (spurious) locally minimum control policy of DP (see Definition \ref{def: local min}) are studied in Section \ref{sec: main result}. Numerical experiments are detailed in Section \ref{sec: examples}. Concluding remarks are drawn in Section \ref{sec: conclusion}.

\textbf{Notation.} Let $\mathbb{R}$ denote the set of real numbers. We use $B(c,r)$ to denote the open ball centered at $c$ with radius $r$ and use $\bar{B}(c,r)$ to denote the closure of $B(c,r)$. The notation $x\in A\setminus B$ means that $x$ is in the set $A$ but not in the set $B$. Let $\|\cdot\|$ denote the Euclidean norm and $\|\cdot\|_F$ denote the Frobenius norm. Let $\nabla_x f(x,y)$ denote the gradient of $f(x,y)$ with respect to $x$ and $\nabla^2_x f(x)$ denote the Hessian of $f(x)$. For the matrix $K$, $K \succ 0$ means that $K$ is positive definite. The notation $C^n$ means that the function is $n$-times continuously differentiable. The notation $\mathbb{E}$ denotes the expectation operator.

\section{Deterministic Problem}
\label{sec:det}

\subsection{Problem Formulation}
\label{sec:prob-form}

Consider a general discrete-time finite-horizon deterministic optimal control problem with $n$ time steps:
\begin{equation}\label{eq: P1}
\tag{P1}
\begin{alignedat}{2}
\min_{u_0,\dots,u_{n-1}\in A} \quad &\mathrlap{\sum_{i=0}^{n-1} c_i(x_i,u_i)+ c_n(x_n)}\\
\nonumber \text{s.t.} \quad & x_{i+1}=f_i(x_i,u_i), & \quad & i=0,\dots,n-1, \\
&x_0 \text{ is given},
\end{alignedat}
\end{equation}
where $x_i\in \mathbb{R}^N$ is the state at time $i$ and $u_i$ is the control input at time $i$ that is constrained to be in an action space $A \subseteq \mathbb{R}^M$. The state transition is governed by the dynamics $f_i:\mathbb{R}^N \times \mathbb{R}^M \to \mathbb{R}^N$. Each time instance $i$ is associated with a stage cost $c_i:\mathbb{R}^N \times \mathbb{R}^M \to \mathbb{R}$ or the terminal cost $c_n:\mathbb{R}^N \to \mathbb{R}$. Given an initial state $x_0$, the goal of the optimal control problem is to find an optimal control input $(u_0,\dots,u_{n-1})$ minimizing the sum of the stage costs and the terminal cost. 
% Notice that the decision variables $x_1,\dots,x_n$ are redundant in the sense that they are all fully determined by $f_i$’s and $u_i$’s.
In this paper, the dynamics $f_i$ and the cost functions $c_i$ are assumed to be at least twice continuously differentiable over $\mathbb{R}^N \times \mathbb{R}^M$, and the action space $A$ is assumed to be compact.

The optimal control problem can be solved by two common approaches. The first approach directly solves \eqref{eq: P1} as a one-shot optimization problem that simultaneously solves for all variables. To simplify the analysis, we eliminate the equality constraints in \eqref{eq: P1} via the notation $C(x_k;u_k,\dots,u_{n-1})$
defined as the cost-to-go started at the time step $k$ with the initial state $x$ and control inputs $u_k,\dots,u_{n-1}$. In other words, 
\begin{align*}
&C(x)=c_n(x), \\
&C(x;u_k,\dots,u_{n-1})=c_k(x,u_k) \\
&\hspace{10em}+C(f_k(x,u_k);u_{k+1},\dots,u_{n-1}),
\end{align*}
for $k=0,\dots,n-1$. The one-shot optimization problem \eqref{eq: P1} can be equivalently written as
\begin{equation}\label{eq:P2}
\tag{P2}
\begin{split}
\min_{u_0,\dots,u_{n-1}\in A} \quad & C(x_0;u_0,\dots,u_{n-1}).
\end{split}
\end{equation}

The second approach to solving the optimal control problem is based on DP. Let $J_k(x_k)$ denote the optimal cost-to-go at the time step $k$ with the initial state $x_k$, \textit{i.e.},
\begin{align*}
J_k(x_k)=\min_{u_k, \dots, u_{n-1}\in A} \quad & C(x_k;u_k,\dots,u_{n-1}) 
% \\ \text{s.t.} \quad & u_i \in A, \quad  i=k,\dots,n-1.
\end{align*}
Then, $J_k$ can be computed in a backward fashion from the time step $n-1$ to time $0$ through the following recursion:
\begin{align}\label{eq:DP}
\nonumber J_n(x)&=c_n(x),\\
J_k(x)&=\min_{u \in A}\{c_k(x,u)+J_{k+1}(f_k(x,u))\}, \tag{P3}
\end{align}
for $k=0,\dots,n-1$. It is worth noting that \eqref{eq:DP} yields a set of \textit{functions} that solve the problem for all initial states, whereas \eqref{eq: P1} produces a \textit{vector} specific to a given $x_0$.
The optimal cost $J_0(x_0)$ equals the optimal objective value of \eqref{eq: P1}.

However, due to the non-convexity of the function, it is generally NP-hard to obtain globally optimal solutions of \eqref{eq:DP} for all states and at all times.  Specifically, when using the DP to solve the optimal control problem \eqref{eq: P1}, the first step is to compute $\min_{u \in A}  \{c_{n-1}(x_{n-1},u)+c_n(f_{n-1}(x_{n-1},u))\}$ for every $x_{n-1} \in \mathbb{R}^N$, which requires solving nonconvex optimization problems if the cost function or the dynamic is nonconvex. Since these intermediate problems are normally solved via local search methods, the best expectation is to obtain a local minimizer for $u_{n-1}$ as a function of $x\in \mathbb{R}^N$, denoted by the policy  $\pi_{n-1}(x)$. As a result, instead of working with truly optimal cost-to-go functions, one may arrive at a sub-optimal cost-to-go at time $n-1$ as follows:
\begin{align*}
J^\pi_{n-1}(x_{n-1}) = c_{n-1}(x_{n-1},&\pi_{n-1}(x_{n-1})) + \\ &c_n(f_{n-1}(x_{n-1},\pi_{n-1}(x_{n-1}))),
\end{align*}
which is obtained based on the local minimizer $\pi_{n-1}(x)$. Subsequently, it is required to solve the optimal decision-making problem $\min_{u \in A} \{c_{n-2}(x_{n-2},u)+J^\pi_{n-1}(f_{n-2}(x_{n-2},u))\}$ for every $x_{n-2} \in \mathbb{R}^N$. By repeating this procedure in a backward fashion  toward the time step 0, we obtain a
group of policy functions $\pi_k$ and sub-optimal cost-to-go functions $J^\pi_k$ for $k=0,\dots,n-1$. Given the initial state $x_0$,
let
\begin{gather*}
u_0^*=\pi_0(x_0), \ x_1^*=f_0(x_0,u_0^*), \
u_1^*=\pi_1(x_1^*), \ x_2^*=f_1(x_1^*,u_1^*) \\
\dots \\
u_{n-1}^*=\pi_{n-1}(x_{n-1}^*), \quad x_n^*=f_n(x_{n-1}^*,u_{n-1}^*),
\end{gather*}
be the control inputs and the states induced by the policies $\pi_0,\dots,\pi_{n-1}$. Then,  $(u_0^*,\dots,u_{n-1}^*)$ is a sub-optimal solution to the original optimal control problem \eqref{eq: P1} with the sub-optimal objective value  $J^\pi_0(x_0)$.
This motivates us to define locally minimum control policies based on solving \eqref{eq:DP} to local optimality.

\begin{definition} \label{def: Q and cost-to-go}
Given a control policy $\pi=(\pi_0,\dots,\pi_{n-1})$,
the associated Q-functions $Q^\pi_k(\cdot,\cdot)$ and cost-to-go functions $J^\pi_k(\cdot)$ under the policy $\pi$ are defined in a backward way from the time step $n-1$ to $0$ through the following recursion:
% \abovedisplayskip
\vspace{-1pt}
\begin{align*} 
&J^\pi_n(x)=c_n(x),\\
&Q^\pi_k(x,u)=c_k(x,u)+J^\pi_{k+1}(f_k(x,u)),\quad  k=0,\dots,n-1, \\
&J^\pi_k(x)=Q^\pi_k(x,\pi_k(x)),\quad k=0,\dots,n-1.
\end{align*}
\end{definition}
\vspace{5pt}
\begin{definition}[local minimizer]\label{def: local min OS}
A vector $(u_0^*,\dots,u_{n-1}^*)$ is said to be a local minimizer of the one-shot optimization problem \eqref{eq:P2} if there exists $\epsilon>0$ such that
\[C(x_0,u_0^*,\dots,u_{n-1}^*) \leq  C(x_0,\Tilde{u}_0,\dots,\Tilde{u}_{n-1})\]
for all $\Tilde{u}_i \in B(u_i^*,\epsilon)\cap A$ where $i=0,\dots, n-1$. It is further called a spurious (non-global) local minimizer of the one-shot optimization problem if $C(x_0,u_0^*,\dots,u_{n-1}^*)>J_0(x_0)$.
\end{definition}
\vspace{3pt}
\begin{definition}[locally minimum control policy] \label{def: local min DP}
A control policy $\pi=(\pi_0,\dots,\pi_{n-1})$
is said to be a locally minimum control policy of DP if for all $k\in \{0,\dots,n-1\}$ and for all $x \in \mathbb{R}^N$, the policy $\pi_k(x)$ is a local minimizer of the Q-function $Q^\pi_k(x,\cdot)$, meaning that there exists $\epsilon_k^\ast(x)>0$ such that 
\[ Q_k^{\pi}(x,\pi_k(x)) \leq Q_k^{\pi}(x,\Tilde{u}) ,\quad \forall \Tilde{u} \in B(\pi_k(x),\epsilon_k^\ast(x))\cap A. \]
It is further called a spurious locally minimum control policy of DP if $J_{0}^{\pi}(x_0)>J_0(x_0)$.
\end{definition}

In the following subsections, we will show that in the deterministic problem, both approaches capture the same local solutions under mild assumptions. 

\subsection{Local minimizers: From DP to one-shot optimization} \label{sec:locminDPOS}

It is well-known that the input sequence induced by a globally minimal control policy is a global minimizer of the one-shot problem \cite{bertsekas1995dynamic}. In this subsection, we will show that the input sequence induced by a spurious locally minimum control policy of DP also corresponds to a spurious local minimizer of the one-shot problem if some mild conditions are satisfied.

\begin{theorem} \label{thm: DP to one-shot epsilon PD}
Assume that $A$ is convex. Consider a (spurious) locally minimum control policy $\pi=(\pi_0,\dots,\pi_{n-1})$, and let the corresponding input and state sequences associated with the initial state $x_0$ be denoted as $(u_0^*,\dots,u_{n-1}^*)$ and $(x_0^*,\dots,x_n^*)$. If $\pi_k$ is twice continuously differentiable in a neighborhood of $x_k^*$ and
$\nabla^2_uQ_k^{\pi}(x_k^*,u_k^*) \succ 0$ for all $k\in\{0,\dots,n-1\}$,
then $(u_0^\ast,\dots,u^\ast_{n-1})$ is also a (spurious) local minimizer of the one-shot problem. 
\end{theorem}
\begin{proof}
% The proof is given in \cite{ding2021analysis} (see Theorem 3).
First, we will use induction to find positive numbers $\delta_0,\dots,\delta_n$ and $\epsilon_0,\dots,\epsilon_{n-1}$ such that
\begin{align}
&\nabla^2_uQ_k^{\pi}(x,u) \succ 0, \label{eq:cond1} \\
&\pi_k(x) \in B(u_k^\ast,\epsilon_k), \label{eq:cond2}\\
&f_{k}(x,u) \in B(x_{k+1}^\ast,\delta_{k+1}), \label{eq:cond3}
\end{align}
for every $x \in B(x_k^\ast,\delta_k)$, $u \in B(u_k^\ast,\epsilon_k) \cap A$, and $k\in\{0,\dots,n-1\}$. At the base step $k=n$, we choose an arbitrary $\delta_n>0$. At the induction step, since $f_k$ is continuous and $\nabla^2_uQ^\pi_k$ is continuous at $(x_k^*,u_k^*)$, there exist $\delta_k>0$ and $\epsilon_k>0$ such that both \eqref{eq:cond1} and \eqref{eq:cond3} are satisfied for all $x \in B(x_k^\ast,\delta_k)$ and $u \in B(u_k^\ast,\epsilon_k) \cap A$. Moreover, as $\pi_k$ is continuous at $x_k^*$, \eqref{eq:cond2} will be satisfied by further reducing $\delta_k$.

For every $(\Tilde{u}_0,\dots,\Tilde{u}_{n-1})$ with $\Tilde{u}_k \in B(u_k^\ast,\epsilon_k) \cap A$, let $(\tilde x_0,\dots,\tilde x_n)$ be its corresponding state sequence (note that $\tilde x_0=x_0$). It follows from \eqref{eq:cond3} that
$\Tilde{x}_k \in B(x_k^\ast,\delta_k)$ for all $k\in\{0,\dots,n-1\},$
which together with \eqref{eq:cond2} implies that
\[
\pi_k(\Tilde{x}_k) \in B(u_k^\ast,\epsilon_k), \quad \forall k\in\{0,\dots,n-1\}.
\]
In light of \eqref{eq:cond1}, $Q_k^{\pi}(\Tilde{x}_k,\cdot)$ is a convex function on the convex set $B(u_k^\ast,\epsilon_k) \cap A$. Because $\pi_k(\Tilde{x}_k) \in B(u_k^\ast,\epsilon_k) \cap A$ is a local minimizer of the function $Q_k^{\pi}(\Tilde{x}_k,\cdot)$, it must be a global minimizer of this function over $B(u_k^\ast,\epsilon_k) \cap A$. Thus, for $k\in\{0,\dots,n-1\}$, we have
\begin{align*}
c_k(\Tilde{x}_k,\Tilde{u}_k)+J^{\pi}_{k+1}(\Tilde{x}_{k+1})=Q_k^{\pi}(\Tilde{x}_k,\Tilde{u}_k)&\geq Q_k^{\pi}(\Tilde{x}_k,\pi_k(\Tilde{x}_k)) \\
&=J^{\pi}_k(\Tilde{x}_k).
\end{align*}
By adding all of the above inequalities, one can obtain
\[
C(x_0;\Tilde{u}_0,\dots,\Tilde{u}_{n-1}) \geq J^{\pi}_0(x_0)=C(x_0;u_0^\ast,\dots,u_{n-1}^\ast),
\]
which shows that $(u_0^\ast,\dots,u_{n-1}^\ast)$ is a local minimizer of the one-shot problem. 
Also, if $\pi$ is a spurious locally minimum control policy of DP, namely, $J_0^\pi(x_0)>J_0(x_0^*)$, then 
\[
C(x_0;u_0^\ast,\dots,u_{n-1}^\ast)=J_0^\pi(x_0)>J_0(x_0).
\]
As a result,  $(u_0^\ast,\dots,u_{n-1}^\ast)$ is also a spurious local minimizer of the one-shot problem.
\end{proof}

\begin{remark}
By taking the contrapositive, one can immediately conclude that the DP method cannot produce any spurious locally minimum control policies that satisfy the regularity conditions in Theorem \ref{thm: DP to one-shot epsilon PD} as long as the one-shot problem has no spurious local minima.
\end{remark}

\subsection{Stationary points: From DP to one-shot optimization} \label{sec:stationaryDPOS}

% If the assumptions of Theorem \ref{thm: DP to one-shot epsilon PD} in Section \ref{sec:locminDPOS} are not satisfied, an induced controlled input of the locally minimum control policy of DP does not necessarily imply a local minimizer of the one-shot problem. Therefore, 
% It is desirable to discover what property an induced controlled input of the locally minimum control policy of DP satisfies for the one-shot problem. 
In this subsection, we will show that the induced controlled input of a locally minimum control policy of DP corresponds to a stationary point of the one-shot problem, under some conditions milder than the assumptions of Theorem \ref{thm: DP to one-shot epsilon PD}.

\begin{definition}\label{def: first order necessary condition}
Given a set ${S}$ and a continuously differentiable function $g$, a
point $s^* \in S$ is said to be a stationary point of the optimization problem $\min_{s\in {S}} g(s)$ if
\begin{align*}
-\nabla_{s}g(s^*) \in \mathcal{N}_{{S}}(s^*),
\end{align*}
where 
$\mathcal{N}_{{S}}(s^*)$ denotes the normal cone of the set ${S}$ at the point $s^*$ \cite{rockafellar2009variational}.
\end{definition}
We branch off into two specific notions of stationarity below.

\begin{definition} [Stationary point]\label{def: stationary point of the one-shot optimization}
    A vector of control inputs $(u_0^*,\dots,u_{n-1}^*)$ is said to be a stationary point of the one-shot optimization if for all $k \in \{0,\dots, n-1 \}$, it holds that $-\nabla_{u_k} C(x_0;u_0^*,\dots,u_{n-1}^*) \in \mathcal{N}_{A}(u_k^*)$.
\end{definition}

\begin{definition} [Stationary control policy] \label{def: stationary control policy of DP}
    A control policy $\pi=(\pi_0,\dots,\pi_{n-1})$ is said to be a stationary control policy of DP if for all $k \in \{0,\dots,n-1\}$ and for all $x \in \mathbb{R}^N$, it holds that $-\nabla_{u} Q_k^{\pi} (x, \pi_k (x)) \in \mathcal{N}_{A}(\pi_k (x))$.
\end{definition}

Now, we will prove that a stationary control policy (which involves a locally minimum control policy) implies a stationary point of the one-shot optimization under mild assumptions.
Let $\mathbf D^\pi_k(x)$ be the Jacobian matrix of $\pi_k(\cdot)$ at point $x$,  $\mathbf D^{f,x}_k(x,u)$ be the Jacobian matrix of the function $f_k(\cdot,u)$ at point $x$ while viewing $u$ as a constant, and $\mathbf D^{f,u}_k(x,u)$ be the Jacobian matrix of $f_k(x,\cdot)$ at point $u$ while viewing $x$ as a constant.

\begin{theorem} \label{thm: stationary}
Consider a stationary control policy $\pi=(\pi_0,\dots,\pi_{n-1})$, and let the associated input and state sequences with the initial state $x_0$ be denoted as $(u_0^*,\dots,u_{n-1}^*)$ and $(x_0^*,\dots,x_n^*)$. If for every $k\in\{0,\dots,n-1\}$:
\begin{enumerate}
\setlength{\itemindent}{-0.5em}
\item $\pi_k$ is continuously differentiable in a neighborhood of $x^*_k$,
\item either $\pi_k(x_k^*)$ is in the interior of $A$ or $\mathbf D^\pi_k(x_k^*)=0$,
\end{enumerate}
then $(u_0^*,\dots,u_{n-1}^*)$ is a stationary point of the one-shot optimization.
\end{theorem}
\begin{proof}
 First, we will apply induction to prove that
\begin{equation}\label{eq:grad}
\nabla_xJ^\pi_k(x_k^*)=\nabla_xC(x_k^*;u_k^*,\dots,u_{n-1}^*)
\end{equation}
holds for $k\in\{0,\dots,n\}$. The base step $k=n$ is obvious. For the induction step, observe that
\begin{align*}
&\nabla_xQ^\pi_k(x,u)=\nabla_xc_k(x,u)+\mathbf D^{f,x}_k(x,u)^T \nabla_xJ^\pi_{k+1}(f_k(x,u)), \\
&\nabla_xJ^\pi_k(x)=\nabla_x[Q^\pi_k(x,\pi_k(x))]
\\&\hspace{13.7mm} = \nabla_xQ^\pi_k(x,\pi_k(x))+\mathbf D^\pi_k(x)^T \nabla_uQ^\pi_k(x,\pi_k(x)).
\end{align*}
Therefore,
\begin{equation}\label{eq:grad1}
\begin{aligned}
\nabla_xJ^\pi_k(x_k^*)=\nabla_xc_k(x_k^*,u_k^*) &+ \mathbf D^{f,x}_k(x_k^*,u_k^*)^T \nabla_xJ^\pi_{k+1}(x_{k+1}^*)\\ &+ \mathbf D^\pi_k(x_k^*)^T \nabla_uQ^\pi_k(x_k^*,u_k^*).
\end{aligned}
\end{equation}
If $u_k^*$ is in the interior of $A$, we have $\nabla_u Q^\pi_k(x_k^*,u_k^*)=0$ by stationarity. Otherwise, by the assumption, we have $\mathbf D^\pi_k(x_k^*)=0$.  In either case, the last term of \eqref{eq:grad1} is zero. Meanwhile,
\begin{align*}
&\nabla_xC(x;u_k^*,\dots,u_{n-1}^*) \\ &= \nabla_xc_k(x,u_k^*)+\nabla_x[C(f_k(x,u_k^*);u_{k+1}^*,\dots,u_{n-1}^*)] \\
&=\nabla_xc_k(x,u_k^*)+\mathbf D^{f,x}_k(x,u_k^*)^T \nabla_xC(f_k(x,u_k^*);u_{k+1}^*,\dots,u_{n-1}^*).
\end{align*}
Now, \eqref{eq:grad} can be obtained by taking $x=x_k^*$ in the above equality and then combining it with the induction hypothesis and \eqref{eq:grad1}.
Finally, for $k\in \{0,\dots,n-1\}$, one can write
\begin{align*}
&\nabla_{u_k}C(x_0;u_0^*,\dots,u_{n-1}^*)\\&=\nabla_uc_k(x_k^*,u_k^*) + \mathbf D^{f,u}_k(x_k^*,u_k^*)^T \nabla_xC(x_{k+1}^*;u_{k+1}^*,\dots,u_{n-1}^*) \\
&=\nabla_uc_k(x_k^*,u_k^*)+\mathbf D^{f,u}_k(x_k^*,u_k^*)^T \nabla_xJ^\pi_{k+1}(x_{k+1}^*) \\&= \nabla_uQ^\pi_k(x_k^*,u_k^*),
\end{align*}
in which the second equality is due to \eqref{eq:grad}. Since $u_k^*$ is a stationary point of $Q^\pi_k(x_k^*,\cdot)$, $-\nabla_u Q^\pi_k(x_k^*,u_k^*) \in \mathcal{N}_{A}(u_k^*)$. Thus, $-\nabla_{u_k}C(x_0;u_0^*,\dots,u_{n-1}^*) \in \mathcal{N}_{A}(u_k^*)$, which proves that $(u_0^*,\dots,u_{n-1}^*)$ is a stationary point of the one-shot optimization.
\end{proof}

\subsection{Local minimizers: From one-shot optimization to DP} \label{sec:locminOSDP}

In this subsection, we will show that each strict local minimizer of the one-shot problem is induced by a locally minimum control policy $\pi$ of DP. Before proving the theorem, we first provide the following useful lemma.

\begin{lemma}\label{prop: strict local min}
Given a function $g: \mathbb{R}^N \times A \to \mathbb{R}$, a point $x^* \in \mathbb{R}^N$ and a number $\epsilon>0$, if $u^* \in A$ is a strict local minimizer of the function $g(x^*,\cdot)$ and $g$ is continuous in a neighborhood of $(x^*,u^*)$, then there exist $\delta>0$ and a function $h: B(x^*,\delta) \to A$ such that $h(x^*)=u^*$ and that the following statements hold for all $x \in B(x^*,\delta)$:
\begin{enumerate}
\item $h(x)$ is a local minimizer of $g(x,\cdot)$.
\item $h(x) \in B(u^*,\epsilon)$.
\item The function $g(x,h(x))$ is continuous at $x$.
\end{enumerate}
\end{lemma}

\begin{proof}
The proof is given in \cite{ding2021analysis} (see Lemma 1).
\end{proof}

\begin{theorem} \label{thm: one-shot to DP}
If the one-shot problem has a (spurious) strict local minimizer $(u_0^*,\dots,u_{n-1}^*)$, then there exists a (spurious) locally minimum control policy $\pi$ of DP with the property that  $\pi_k(x_k^*)=u_k^*$ for all $k\in\{0,\dots,n-1\}$, where $(x_0^*,\dots,x_n^*)$ is the state sequence associated with the (spurious) solution of the one-shot problem.
% If the one-shot problem has a strict local minimizer (resp. spurious strict
% local minimizer) $(u_0^*,\dots,u_{n-1}^*)$, then there exists a locally minimum control policy (resp. spurious locally minimum control policy) $\pi$ of DP with the property that  $\pi_k(x_k^*)=u_k^*$ for all $k\in\{0,\dots,n-1\}$, where $(x_0^*,\dots,x_n^*)$ is the state sequence associated with the (spurious) solution of the one-shot problem.
\end{theorem}

\begin{proof}
Let $(u_0^*,\dots,u_{n-1}^*)$ be a strict local minimizer of the one-shot problem. There exists $\epsilon>0$ such that
\begin{equation}\label{eq:strictlocalmin}
C(x_0;u_0^*,\dots,u_{n-1}^*)<C(x_0;u_0,\dots,u_{n-1}),
\end{equation}
for every control sequence $(u_0,\dots,u_{n-1}) \neq (u_0^*,\dots,u_{n-1}^*)$ with the property that $u_i \in B(u_i^*,\epsilon) \cap A$ for $i=0,\dots,n-1$. In what follows, we will prove by a backward induction that there exist policies $\pi_0,\dots,\pi_{n-1}$, positive numbers $\delta_0,\dots,\delta_{n}$, and corresponding cost-to-go functions $J_0^{\pi},\dots,J_n^{\pi}$ such that they jointly satisfy the following properties:
\begin{enumerate}
\item $\pi_k(x_k)$ is a local minimizer of the function $Q^\pi_k(x_k,\cdot)$ for all $x_k \in \mathbb{R}^N$.
\item $\pi_k(x_k^*)=u_k^*$.
\item For all $x_k \in B(x_k^*,\delta_k)$, it holds that
\begin{align*}
\pi_k(x_k) \in B(u_k^*,\epsilon), ~~ f_k(x_k,\pi_k(x_k)) \in B(x_{k+1}^*,\delta_{k+1}).
\end{align*}
\item $J^\pi_k$ is lower semi-continuous on $\mathbb{R}^N$ and continuous on $B(x_k^*,\delta_k)$.
\end{enumerate}

For the base step $k=n$, we choose an arbitrary $\delta_n>0$ and notice that $J^\pi_n(x)=c_n(x)$, implying that $J^\pi_n$ is always continuous.  For $k<n$, assume that $\pi_{k+1},\dots,\pi_{n-1}$ and $\delta_{k+1},\dots,\delta_n$ with the above properties have been found. 
% Here, notice that when $k=n-1$, we assume that only $\delta_n$ has been found.

First, by the continuity of $f_k$, there exist $\delta'_k>0$ and $0<\epsilon_k<\epsilon$ such that
\begin{equation}\label{eq:fcont}
f_k(x_k,u_k) \in B(x_{k+1}^*,\delta_{k+1}), \quad \forall (x_k,u_k) \in S_k,
\end{equation}
where $S_k=B(x_k^*,\delta'_k) \times (B(u_k^*,\epsilon_k) \cap A).$

Since $Q^\pi_k(x_k,u_k)=c_k(x_k,u_k)+J^\pi_{k+1}(f_k(x_k,u_k))$ and $J^\pi_{k+1}$ is continuous on $B(x_{k+1}^*,\delta_{k+1})$, $Q^\pi_k$ is continuous on $S_k$. Next, for every $\tilde u_k \in B(u_k^*,\epsilon_k) \cap A$, if we define
\begin{gather*}
\tilde x_{k+1}=f_k(x_k^*,\tilde u_k), \quad \tilde u_{k+1}=\pi_{k+1}(\tilde x_{k+1}), \\
\tilde x_{k+2}=f_{k+1}(\tilde x_{k+1},\tilde u_{k+1}), \quad \tilde u_{k+2}=\pi_{k+2}(\tilde x_{k+2}),  \\
\dots \\
\tilde x_{n-1}=f_{n-2}(\tilde x_{n-2},\tilde u_{n-2}), \quad \tilde u_{n-1}=\pi_{n-1}(\tilde x_{n-1}),
\end{gather*}
by applying \eqref{eq:fcont} and then the third property above repeatedly, we arrive at
\[
\tilde u_i \in B(u_i^*,\epsilon) \cap A, \quad \forall i\in \{k+1,\dots,n-1\}.
\]
When $\tilde u_k \neq u_k^*$, it follows from \eqref{eq:strictlocalmin} and the second property above that
\begin{align*}
Q^\pi_k(x_k^*&,\tilde u_k)=C(x_k^*;\tilde u_k,\dots,\tilde u_{n-1}) 
\\&=C(x_0;u_0^*,\dots,u_{k-1}^*,\tilde u_k,\dots,\tilde u_{n-1})-\sum_{i=0}^{k-1}c_i(x_i^*,u_i^*) \\
&>C(x_0;u_0^*,\dots,u_{n-1}^*)-\sum_{i=0}^{k-1}c_i(x_i^*,u_i^*) \\ &=C(x_k^*;u_k^*,\dots,u_{n-1}^*)=Q^\pi_k(x_k^*,u_k^*).
\end{align*}
As a result, $u_k^*$ is a strict local minimizer of $Q^\pi_k(x_k^*,\cdot)$. Applying Lemma~\ref{prop: strict local min} to the function $Q^\pi_k$ with $x_k^*$ and $\epsilon_k$, one can find $0<\delta_k<\delta_k'$ and a function $h_k:B(x_k^*,\delta_k) \to A$ such that $h_k(x_k^*)=u_k^*$ and that the following statements hold for every $x_k \in B(x_k^*,\delta_k)$:
\begin{enumerate}
\item $h_k(x_k)$ is a local minimizer of $Q^\pi_k(x_k,\cdot)$.
\item $h_k(x_k) \in  B(u_k^*,\epsilon_k) \subseteq B(u_k^*,\epsilon)$, which together with \eqref{eq:fcont} implies that \\ $f_k(x_k,h_k(x_k)) \in B(x_{k+1}^*,\delta_{k+1})$.
\item The function $Q^\pi_k(x_k,h_k(x_k))$ is continuous at $x_k$.
\end{enumerate}

Let $\pi_k$ be the extension of the function $h_k$ by setting $\pi_k(x_k)$ to be any global minimizer of the lower semi-continuous function $Q^\pi_k(x_k,\cdot)$ over the compact set $A$ if $x_k \notin B(x_k^*,\delta_k)$. 
Obviously, $\pi_k$ satisfies the first three properties. To verify the last property, observe that 
\[
J^\pi_k(x_k)=\begin{dcases*}
Q^\pi_k(x_k,h_k(x_k)), & if $x_k \in B(x_k^*,\delta_k)$, \\
H_k(x_k), & otherwise,
\end{dcases*}
\]
\noindent in which $H_k(x_k)=\min_{u \in A}Q^\pi_k(x_k,u)$, and therefore $J^\pi_k$ is continuous on the set $B(x_k^*,\delta_k)$. In addition, note that $J_{k+1}^\pi$ and thus $Q^\pi_k$ is lower semi-continuous, while $A$ is compact. Hence, it follows from the Berge maximum theorem \cite{berge1997topological} that $H_k$ is also lower semi-continuous on $\mathbb{R}^N$, which implies that $J^\pi_k$ is lower semi-continuous on  $\mathbb{R}^N \setminus\bar B(x_k^*,\delta_k)$.
For every point $\bar x_k$ on the boundary of $B(x_k^*,\delta_k)$, since $H_k$ is lower semi-continuous at $\bar x_k$, for every $\bar\epsilon>0$ there exists $\bar\delta>0$ such that
\[
J^\pi_k(x_k) \geq H_k(x_k)>H_k(\bar x_k)-\bar\epsilon=J^\pi_k(\bar x_k)-\bar\epsilon
\]
holds for all $x_k \in B(\bar x_k,\bar\delta)$. Therefore, $J^\pi_k$ is also lower semi-continuous at $\bar x_k$.

By the first and second properties, $\pi=(\pi_0,\dots,\pi_{n-1})$ is a locally minimum control policy of DP. Also, if $(u_0^*,\dots,u_{n-1}^*)$ is a spurious local minimizer of the one-shot problem, then 
$J^\pi_0(x_0)=C(x_0;u_0^*,\dots,u_{n-1}^*)>J_0(x_0),$
which implies that $\pi$ is also a spurious locally minimum control policy of DP.
\end{proof}

\begin{remark}
Theorem \ref{thm: one-shot to DP} shows that, under mild conditions, DP is a reformulation from a single one-shot optimization problem to a sequence of optimization problems that preserves local minimizers. 
By taking the contrapositive of Theorem \ref{thm: one-shot to DP}, one can immediately obtain the result that the one-shot problem has no spurious strict local minimizers as long as DP has no spurious locally minimum control policies.
\end{remark}

% \begin{remark}
% DP can be viewed as a reformulation of the optimal control problem from a single one-shot optimization problem to a sequence of optimization problems. 
% % When a non-convex problem is reformulated, its local minimizers could change and for example convexification often serves as a reformulation in a higher-dimensional space that eliminates spurious solutions. However, 

% % Note that a spurious solution of DP is a set of functions, where a spurious solution of the one-shot optimization is a vector. 
% \end{remark}

\begin{remark}\label{continuousinfinite}
Pontryagin's minimum principle implies that a global minimizer of the one-shot problem achieves a global optimality of each DP problem minimizing Hamiltonian. One can restrict the domain to apply the principle to a local minimizer of the one-shot problem; it achieves a local optimality of $u_k^*$ for each DP problem if $J_k^\pi$ are evaluated at the associated state $x_k^\star$. Theorem \ref{thm: one-shot to DP} is a generalization of Pontryagin's principle in the sense that from each local minimizer of the one-shot problem, we obtain a locally minimum control policy instead of $u_k^*$; \textit{i.e.}, a set of \textit{functions} that achieves a local optimality of every DP problem for all $x\in \mathbb{R}^N$. We further require a ``strict" local minimizer of the one-shot problem to ensure that a local optimality is obtained at all points in the neighborhood of $x_k^\star.$
Meanwhile, one can now anticipate that Theorem \ref{thm: DP to one-shot epsilon PD} would correspond to the converse of Pontryagin's principle. The principle provides sufficient conditions for the one-shot problem if we have a convex action space, convex cost functions, and linear dynamics \cite{bertsekas1995dynamic, bressan2007intro}. 
In contrast, Theorem \ref{thm: DP to one-shot epsilon PD} assumes a convex action space but still has general nonlinear transition dynamics. Theorem \ref{thm: DP to one-shot epsilon PD} instead requires a locally ``strictly" convex Q-functions (Hamiltonian) for each DP sub-problem. The connection between our results and Pontryagin's principle suggests the possibility of the extension of the above results to the continuous-time setting.
\end{remark}

\begin{remark}\label{allresult}
In fact, all results of our paper can be naturally generalized to the continuous-time setting, but the analysis is left as future work due to space restrictions. To outline the pathway for generalization, note that the Hamilton-Jacobi-Bellman equation for a given continuous-time system can be obtained from developing a discrete-time model, obtaining the Bellman equation for that model, and then closing the gap between the continuous-time and discrete-time system via taking a limit \cite{bertsekas1995dynamic}. 
% This approach allows us to naturally generalize all results of the paper to the continuous-time setting.
Moreover, the infinite-horizon case is also treated in \cite{bertsekas1995dynamic} as the stationary limit of a finite-horizon problem, which again allows us to extend our results to the infinite-horizon case.
\end{remark}

Considering Theorems \ref{thm: DP to one-shot epsilon PD} and \ref{thm: one-shot to DP} altogether, one can conclude that under mild conditions, each local minimizer of the one-shot optimization corresponds to some control input induced by a locally minimum control policy, and vice versa.

\subsection{Numerical Examples} \label{sec: detexamples}
To effectively demonstrate the results of this section via visualization, we will provide two low-dimensional examples. 

\textit{Example 1}: Consider an optimal control problem {with the control constraint $A=[-10,10]$ and}
\begin{align*}
  & c_0(x,u)=0, \\
  &c_1(x,u)= \frac{1}{4}u^4-\frac{3x+4}{3}u^3 +\frac{3x^2+8x+3}{2}u^2\\
  &\hspace{1.4cm}-x(x+1)(x+3)u+\exp{(x^4)},\\
  & c_2(x)=0,\ f_0(x,u)=x+u, \ f_1(x,u)=x+u.
\end{align*}
At the initial state $x_0=0$, the one-shot problem is written as 
\begin{align*}
    \min_{u_0 \in A,u_1 \in A} \ & \Big\{\frac{1}{4}u_1^4-\frac{3u_0+4}{3}u_1^3 +\frac{3u_0^2+8u_0+3}{2}u_1^2\\
  &-u_0(u_0+1)(u_0+3)u_1+\exp{(u_0^4)} \Big\}.
\end{align*}
This one-shot optimization problem has 3 spurious local minimizers $(-0.523,-0.523), (-0.523,2.477), (0.938,0.938)$ and the globally optimal minimizer $(0.938,3.938)$. The landscape of this objective function is shown in Fig.~\ref{fig: no bifurcate landscape}.

\begin{figure}[t]
\centering \subfloat[Example 1]{\label{fig: no bifurcate landscape}\includegraphics[width=43mm]{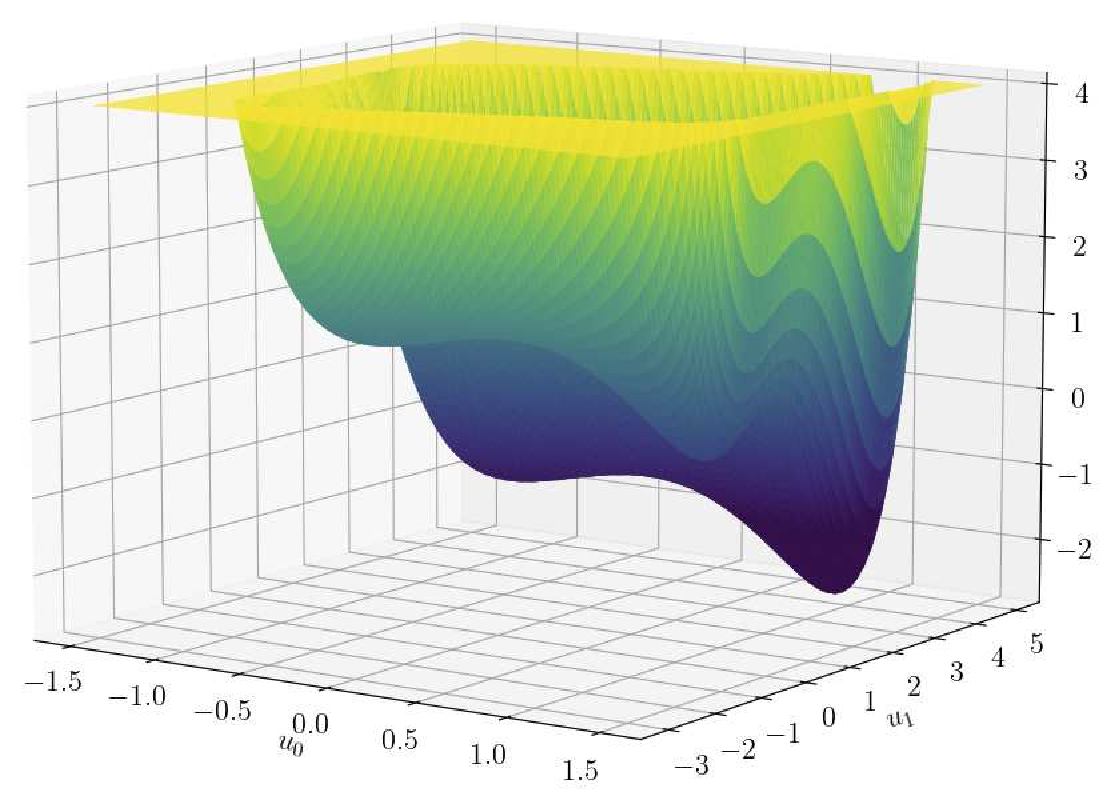}} \hspace{1mm}\subfloat[Example 2]{\label{fig:bifurcate landscape}\includegraphics[width=43mm]{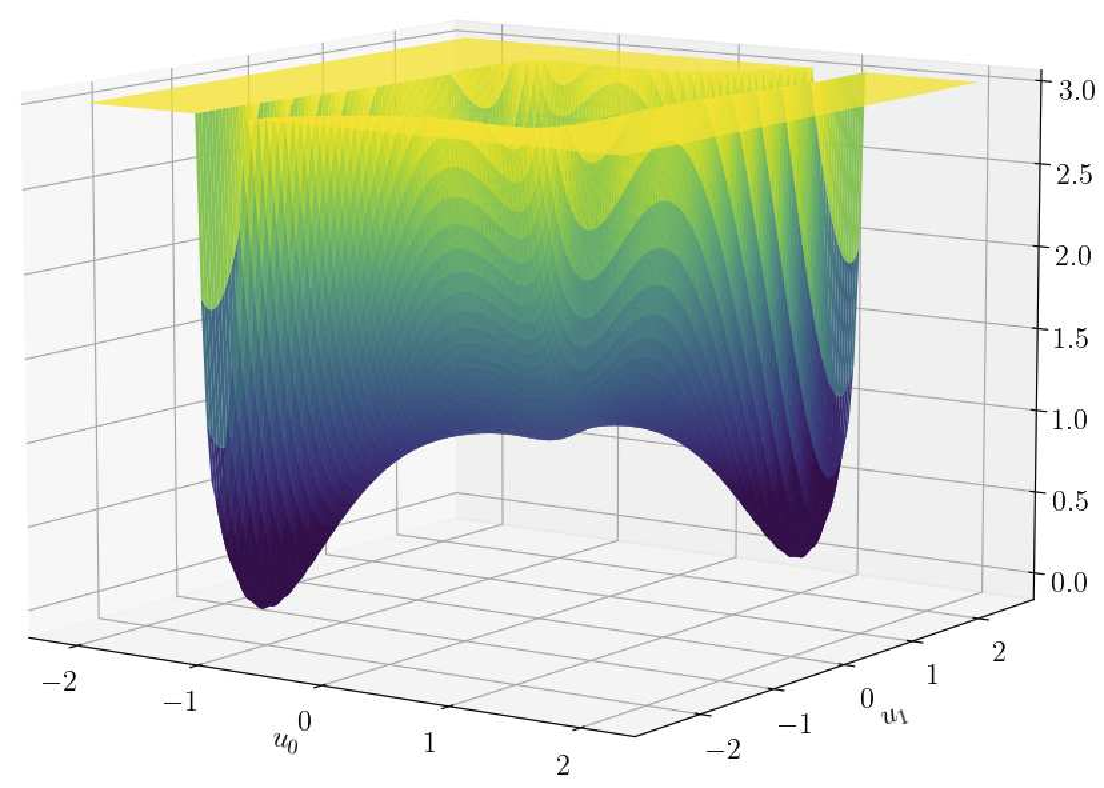}} \caption{Landscape of the one-shot optimization: (a) Each local minimizer is equivalent to a set of control inputs induced by each locally minimum control policy. (b) $(0,0)$ is a control input induced by a locally minimum control policy but not a local minimizer of the one-shot optimization. However, it is indeed a stationary point of the one-shot optimization.}
\label{fig:Landscape}
\vspace{-1mm}
\end{figure}

The optimal control problem can also be solved sequentially by DP. At the time step $1$, the Q-function is $Q_1^\pi(x,u_1)=c_1(x,u_1)$,
which has the maximum point $x+1$, the spurious local minimizer $x$ and the global minimizer $x+3$. One can choose between the two different continuous policies 
\[\pi_1(x)=\begin{cases} 
      x, & |x| \leq 10,\\
      10\cdot \text{sgn}(x), &\text{otherwise},
   \end{cases} \]
or
\[\pi_1(x)=\begin{cases}
      x+3, & -13\leq x \leq 7,\\
      10\cdot \text{sgn}(x), & \text{otherwise},
   \end{cases}\]
where $\text{sgn}(x)$ denotes the sign of $x$. The first policy has the cost-to-go function $J_1^\pi (x) = -\frac{1}{12}(3x^4+16x^3+18x^2)+\exp(x^4)$ for $|x|\leq 10$ and the second policy has $J_1^\pi (x) = -\frac{1}{12}(3x^4+16x^3+18x^2+27)+\exp(x^4)$ for $-13\leq x \leq 7$.

At the time step $0$ and the initial state $x_0=0$, the Q-function is  $Q_0^\pi(0,u_0)=J_1^\pi (u_0)$. For the first policy, the Q-function has a spurious local minimizer at $u_0=-0.523$ and a global minimum at $u_0=0.938$. If we choose $\pi_0(0)=-0.523$, then the induced input under $\pi$ of DP is $(-0.523,-0.523)$ and if we choose $\pi_0(0)=0.938$, then the induced input under $\pi$ of DP is $(0.938,0.938)$. 
Both of these input sequences are spurious local minimizers of the one-shot problem.

The Q-function of the second policy has a spurious local minimizer at $u_0=-0.523$ and a global minimum at $u_0=0.938$. If we choose $\pi_0(0)=0.938$, then the locally minimum control policy $\pi$ is non-spurious and its induced input $(0.938,3.938)$ is the global minimizer of the one-shot problem. However, if we choose $\pi_0(0)=-0.523$, then $\pi$ is spurious and its induced input $(-0.523,2.477)$ is the spurious minimizer of the one-shot problem.

In this example, one can observe that each strictly local minimizer of the one-shot problem corresponds to a locally minimum control policy of DP, which validates the result of Theorem \ref{thm: one-shot to DP}. In addition, it can be noticed that since $\nabla_u^2 Q_0^\pi(0,-0.523)$ and $\nabla_u^2 Q_0^\pi(0,0.938)$ are both strictly positive for each of the two policies, Theorem \ref{thm: DP to one-shot epsilon PD} clearly holds.

% This can be found by calculating 
% The details on the calculation of the Q-function are given in \cite{supplementary} (see Section \nbRoman{2}-E). 

\textit{Example 2}: Consider the problem in Example 1 but change $c_1(x,u)$ to $\frac{1}{4}u^4-\frac{x}{3}u^3 -x^2u^2+\exp{(x^4)}$. At the initial state $x_0=0$, the one-shot problem can be written as 
\begin{align*}
    \min_{u_0 \in A,u_1 \in A} \Big\{\frac{1}{4}u_1^4-\frac{u_0}{3}u_1^3 -u_0^2u_1^2+\exp{(u_0^4)}\Big\}.
\end{align*}
It has 3 stationary points $(0,0)$ and $((\log (\frac{8}{3}))^{\frac{1}{4}},2(\log (\frac{8}{3}))^{\frac{1}{4}})$ and $(-(\log (\frac{8}{3}))^{\frac{1}{4}},-2(\log (\frac{8}{3}))^{\frac{1}{4}})$. The latter two are the global minimizers of this one-shot problem. 
For $(0,0)$, we take $u_0=u_1=\epsilon$ and use the Taylor expansion of the exponential function to arrive at $\frac{1}{4}\epsilon^4-\frac{1}{3}\epsilon^4 -\epsilon^4+\exp{(\epsilon^4)}
%  =&-\frac{13}{12}\epsilon^4+\exp{\epsilon^4}\\
%  =&-\frac{13}{12}\epsilon^4+1+\epsilon^4+o(\epsilon^4)   \\
 =-\frac{1}{12}\epsilon^4+1+o(\epsilon^4),$
which is strictly less that $1$ for sufficiently small values of $\epsilon$. This implies that $(0,0)$ is not a local minimizer of the one-shot problem. The landscape of this objective function is shown in Fig.~\ref{fig:bifurcate landscape}.
% \begin{figure}[t]
% \centering
% \includegraphics[width=0.8 \linewidth]{AISTATS2020/bifurcate_example.png}
% \caption{Landscape of the one-shot problem in Example 2.}
% % \caption{Landscape of $\frac{1}{4}u_1^4-\frac{u_0}{3}u_1^3 -u_0^2u_1^2+\exp{u_0^4}$.}
% \label{fig:bifurcate landscape}
% \end{figure}
It can also be solved sequentially by DP. For the initial state $x_0$, it has 3 different induced input sequences under the locally minimum control policy: $(\log (\frac{8}{3}))^{\frac{1}{4}},2(\log (\frac{8}{3}))^{\frac{1}{4}})$, $(-(\log (\frac{8}{3}))^{\frac{1}{4}},-2(\log (\frac{8}{3}))^{\frac{1}{4}})$ and $(0,0)$. The first two points are the global minimizers of the one-shot problem but $(0,0)$ is not a local minimizer of the one-shot problem.

In this example, $\nabla^2_uQ_1^\pi(0,0) =  \nabla^2_u c_1(0,0) =0$
% has 3 stationary points $0,-x,2x$ and all 3 points will merge to a single point when $x=0$.
violates the assumptions in Theorem \ref{thm: DP to one-shot epsilon PD},
% (the radius of the region in which $u_1=-x$ or  $u_1=2x$ is a local minimizer is discontinuous at $x=0$, thus $\inf \epsilon_1^\ast(x)=0$ and $\nabla^2_u Q_1^\pi(0,0)=0$) 
and thus $(0,0)$ is not a local minimizer of the one-shot problem. This clarifies the role of the regularity conditions needed in the theorem. On the other hand, $Q_1^\pi(x,\cdot)$ has three stationary control policies $0, -x, 2x$.
Consistent with Theorem \ref{thm: stationary}, $(0,0)$ is a saddle point (which is a stationary point) of the one-shot optimization.

\section{Deterministic Problem under a Parameterized policy}
\label{sec:detparam}

\subsection{Problem Formulation}\label{detparam prob-form}

In Section \ref{sec:det}, the one-shot optimization approach is referred to as an open-loop control, in the sense that it determines all the control inputs at once, only given an initial state. On the other hand, the dynamic programming approach is referred to as a closed-loop control, in the sense that the control input of each time step is the function of the output of the previous step \cite{bertsekas1995dynamic}. In this section, we formulate both approaches to a closed-loop control. To achieve this, we can replace the control inputs of the one-shot optimization with a parameterized policy. We still optimize over a vector at once, which means that it can be solved in a one-shot fashion. However, this method becomes a type of closed-loop control in the sense that a function of both the parameters at each step and the output of the previous step determines the control input \cite{kumar1986stochastic}. Also, it is reasonable to adopt such parameterized policies for dynamic programming as well, which would still be a closed-loop control. Note that both approaches now optimize over a set of parameters so that they can be directly compared in terms of their landscapes. This motivates us to modify Definitions \ref{def: Q and cost-to-go}, \ref{def: local min OS}, \ref{def: local min DP}, \ref{def: stationary point of the one-shot optimization}, and \ref{def: stationary control policy of DP} to incorporate parameterized policies.

\begin{definition} \label{def: param policy}
Given a parameter space $\Theta$ and a compact action space $A$, let $\mu_\theta(\cdot):\mathbb{R}^N \to A$ be a bounded real-valued function parameterized by $\theta \in \Theta$, which satisfies the continuity assumption that for all $\epsilon >0$, there exists $\delta>0$ such that
\begin{align}\label{continuity}
\| \theta - \theta'\|<\delta& \Rightarrow \sup_{x\in \mathbb{R}^N} \|\mu_{\theta}(x)-\mu_{\theta'}(x)\|<\epsilon.
\end{align}
% \vspace{0.01mm}
\end{definition}

Now, we modify the deterministic problems (\ref{eq: P1}), (\ref{eq:P2}), and (\ref{eq:DP}) to a discrete-time finite-horizon deterministic optimal control problem under a parameterized policy as follows:
\begin{equation}\label{eq:DP1}
\tag{PP1}
\begin{alignedat}{2}
\min_{\theta_0,\dots,\theta_{n-1}\in\Theta} \quad &\mathrlap{\sum_{i=0}^{n-1} c_i(x_i,\mu_{\theta_i}(x_i))+ c_n(x_n)}\\
\nonumber \text{s.t.} \quad & x_{i+1}=f_i(x_i,\mu_{\theta_i}(x_i)), & \quad & i=0,\dots,n-1, \\
&x_0 \text{ is given}.
\end{alignedat}
\end{equation}

% Notice that the decision variables $x_1,\dots,x_n$ are redundant in the sense that they are all fully determined by $f_i$'s and $\theta_i$'s.

\begin{definition} \label{def: detparam Q and cost-to-go}
Given a control policy parameter vector $\pi=(\theta_0,\dots,\theta_{n-1})$,
the associated Q-functions $Q^\pi_k(\cdot,\cdot)$ and cost-to-go functions $J^\pi_k(\cdot)$ under the policy $\pi$ are defined in a backward way from the time step $n-1$ to the time step $0$ through the following recursion:
\begin{align*} 
&J^\pi_n(x)=c_n(x),\\
&Q^\pi_k(x,\mu_{\theta}(x))=c_k(x,\mu_{\theta}(x))+J^\pi_{k+1}(f_k(x,\mu_{\theta}(x))),\\ &\hspace{60mm} k=0,\dots,n-1, \\
&J^\pi_k(x)=Q^\pi_k(x,\mu_{\theta_k}(x)),\quad  k=0,\dots,n-1.
\end{align*}
\end{definition}
\vspace{3pt}
Then, the one-shot optimization problem \eqref{eq:DP1} can be equivalently written as
\begin{equation}\label{eq:DP2}
\tag{PP2}
\begin{split}
\min_{\pi=(\theta_0,\dots,\theta_{n-1}) \in \Theta^n} \quad & J_0^\pi (x_0)  
\end{split}
\end{equation}
and DP approach can be written as the following backward recursion:
\begin{equation}\label{eq:DDP}
\tag{PP3}
\begin{split}
J_n(x)&=c_n(x),\\
J_k(x)&=\min_{\theta \in \Theta}\{c_k(x,\mu_{\theta}(x))+J_{k+1}(f_k(x,\mu_{\theta}(x)))\},
\end{split}
\end{equation}
for $k=0,\dots,n-1$. Note that $\pi$ was previously defined as a control policy $(\pi_0,..,\pi_{n-1})$, but we use the equivalent definition $(\theta_0,...,\theta_{n-1})$ in the parameterized case. We also call it \textit{control policy parameter vector} alternatively. 

\begin{definition}[local minimizer of the one-shot optimization]\label{def: detparam local min OS}
A control policy parameter vector $\pi=(\theta_0^*,\dots,\theta_{n-1}^*)$ is said to be a local minimizer of the one-shot optimization if there exists $\epsilon>0$ such that
\[J_0^{\pi}(x_0) \leq  J_0^{\Tilde{\pi}}(x_0)\]
for all $\Tilde{\pi}=(\Tilde{\theta}_0, \dots, \Tilde{\theta}_{n-1}) \in (B(\theta_0^*,\epsilon)\cap \Theta) \times \dots \times (B(\theta_{n-1}^*,\epsilon)\cap \Theta)$. 
% It is further called a spurious local minimizer of the one-shot optimization problem if $J_{0}^{\pi}(x_0)>J_0(x_0)$.
\end{definition}

\begin{definition}[local minimizer of DP] \label{def: detparam local min DP}
A control policy parameter vector $\pi=(\theta_0^*,\dots,\theta_{n-1}^*)$
is said to be a local minimizer of DP if for all $k\in \{0,\dots,n-1\}$ and for all $x \in \mathbb{R}^N$, the policy parameter $\theta_k^*$ is a local minimizer of $Q^\pi_k(x,\mu_{(\cdot)} (x))$, meaning that there exists $\epsilon_k^\ast>0$ such that
% \abovedisplayskip
\vspace{-2pt}
\begin{equation}\label{def: param space}
    Q_k^{\pi}(x,\mu_{\theta_k^*}(x)) \leq Q_k^{\pi}(x,\mu_{\Tilde{\theta}}(x)) ,\quad \forall \Tilde{\theta} \in B(\theta_k^*,\epsilon_k^\ast)\cap \Theta.
\end{equation}
% It is further called a spurious local minimizer of DP if $J_{0}^{\pi}(x_0)>J_0(x_0)$.
\end{definition}
\vspace{1pt}
\begin{definition}[Stationary point of the one-shot optimization]\label{def: detparam stationary point of the one-shot optimization}
A control policy parameter vector $\pi=(\theta_0^*,\dots,\theta_{n-1}^*)$ is said to be a stationary point of the one-shot optimization if for all $k \in \{0,\dots,n-1\}$, it holds that $-\nabla_{\theta_k}J_0^\pi (x_0) \in \mathcal{N}_{\Theta}(\theta_k^*)$.
\end{definition}

\begin{definition}[Stationary point of DP] \label{def: detparam stationary point of DP}
A control policy parameter vector $\pi=(\theta_0^*,\dots,\theta_{n-1}^*)$ is said to be a stationary point of DP if for all $k \in \{0,\dots,n-1\}$ and for all $x \in \mathbb{R}^N$, it holds that $-\nabla_{\theta_k} Q_k^\pi (x, \mu_{\theta_k^*}(x)) \in \mathcal{N}_{\Theta}(\theta_k^*)$.
\end{definition}

\begin{remark}
    By comparing (\ref{eq:P2}) with (\ref{eq:DP2}) as well as comparing Definition \ref{def: local min OS} with Definition \ref{def: detparam local min OS}, notice that one-shot optimization now considers $J_0^\pi (x_0)$ instead of $C(x_0; \theta_0, \dots, \theta_{n-1})$, since the two definitions are equivalent when the parameterized policy is incorporated.
\end{remark}

We can compare Definition \ref{def: detparam local min DP} with the following definition: 
\begin{equation}\label{def: action space}
 \begin{aligned}
&\forall k \in \{0,\dots,n-1\},\ \forall x \in \mathbb{R}^N, \
 \exists\epsilon_k^\ast(x)>0 \ \text{such that}\\ &Q_k^{\pi}(x,\mu_{\theta_k^*}(x)) \leq Q_k^{\pi}(x,\tilde{u}) ,\quad \forall \tilde{u} \in B(\mu_{\theta_k^*}(x), \epsilon_k^{*} (x)) \cap A, 
 \end{aligned}
 \end{equation}

Definition \ref{def: detparam local min DP} considers the open ball centered at the policy parameter in the parameter space, while (\ref{def: action space}) considers the corresponding open ball 
% centered at the evaluated action 
in the action space. Proposition \ref{prop: action and parameter} establishes the relationship between these definitions.

\begin{proposition}\label{prop: action and parameter}
If an arbitrary control policy parameter vector $\pi=(\theta_0^*, \dots, \theta_{n-1}^*)$ satisfies (\ref{def: action space}) with $\inf_{x\in \mathbb{R}^N} \epsilon_k^*(x) >0$ for all $k\in \{0,\dots,n-1\}$, then it is a local minimizer of DP.
\end{proposition}

\begin{proof}
Since $\inf_{x\in \mathbb{R}^N} \epsilon_k^*(x) >0$, by the continuity assumption, for every $k\in \{0,\dots,n-1\}$, there exists $\delta_k>0$ such that 
\begin{align*}
\| \theta - \theta_k^*\|<\delta_k \Rightarrow \sup_{x\in \mathbb{R}^N} \|\mu_{\theta}(x) - \mu_{\theta_k^*}(x)\|<\inf_{x\in \mathbb{R}^N} \epsilon_k^*(x)
\end{align*}

That is, for all $\theta \in B(\theta_k^*, \delta_k) \cap \Theta$, $\|\mu_{\theta}(x) - \mu_{\theta_k^*}(x)\|<\epsilon_k^*(x)$ for all $x \in \mathbb{R}^N$. Notice that Definition \ref{def: param policy} implies that $\mu_{\theta}(x) \in A$ for all $x \in \mathbb{R}^N$. Thus, it holds for all $x \in \mathbb{R}^N$ that
\[
\theta \in B(\theta_k^*, \delta_k) \cap \Theta \Rightarrow \mu_{\theta}(x) \in B(\mu_{\theta_k^*}(x), \epsilon_k^*(x)) \cap A.
\]

Thus, given a control policy parameter vector satisfying (\ref{def: action space}), for all $k \in \{0, \dots, n-1\}$ and for all $x\in\mathbb{R}^N$, (\ref{def: param space}) holds if one substitutes $\epsilon_k^*$ with $\delta_k$. This completes the proof.
\end{proof}

\begin{remark}
The converse of Proposition \ref{prop: action and parameter} does not hold. For example, suppose there exists $\epsilon_k^* > 0$ such that $\mu_\theta(x)$ takes the same value for all $\theta \in B(\theta_k^*, \epsilon_k^*) \cap \Theta$. While this control policy satisfies the continuity assumption, $\theta_k^*$ is clearly a local minimizer of DP, which satisfies (\ref{def: param space}). However, it is even possible that $\mu_{\theta_k^*}(x)$ is a strict local maximizer of $Q_k^\pi(x,\cdot)$.

% Recall from Remark \ref{epsilon 0 remark} that (\ref{epsilon 0}) was necessary for Theorem \ref{thm: DP to one-shot epsilon}. Also, 

Note that the condition $\inf_{x\in \mathbb{R}^N} \epsilon_k^*(x) >0$ is necessary for Proposition \ref{prop: action and parameter}. Thus, the proposition implies that if we use our notion of a local minimizer of DP, we no longer need to assume $\inf_{x\in \mathbb{R}^N} \epsilon_k^*(x)>0$ while establishing the relationship from DP to one-shot optimization, which was the case in the (non-parameterized) deterministic case presented in our conference paper (see Theorem 2 in \cite{ding2021analysis}).
\end{remark}

\subsection{From DP to one-shot optimization}\label{detparam DPOS}

In this subsection, we will show that in the deterministic case with a parameterized policy, each local minimizer (stationary point) of DP directly corresponds to some local minimizer (stationary point) of the one-shot optimization. 

\begin{theorem}\label{thm: detparam loc min}
Consider a local minimizer of DP $\pi = (\theta_0^*, ..., \theta_{n-1}^*)$. Then, $\pi$ is also a local minimizer of the one-shot optimization.
\end{theorem}

\begin{proof}
Since $(\theta_0^*, ..., \theta_{n-1}^*)$ is a local minimizer of DP, there exist $\epsilon_0^*, \dots, \epsilon_{n-1}^* > 0$ such that 
\begin{align*}
J_0^\pi(x_0)&=Q_0^\pi (x_0, \mu_{\theta_0^*}(x_0)) \leq Q_0^\pi (x_0, \mu_{\tilde{\theta}_0}(x_0)) \\ 
&= c_0(x_0, \mu_{\tilde{\theta}_0}(x_0)) + Q_1^\pi (\tilde{x}_1, \mu_{\theta_1^*}(\tilde{x}_1)) \\ & \hspace{44mm} (\tilde{x}_1=f_0(x_0, \mu_{\tilde{\theta}_0}(x_0))) \\ 
&\leq c_0(x_0, \mu_{\tilde{\theta}_0}(x_0)) + Q_1^\pi (\tilde{x}_1, \mu_{\tilde{\theta}_1}(\tilde{x}_1)) \\ 
&= c_0(x_0, \mu_{\tilde{\theta}_0}(x_0)) +  c_1(\tilde{x}_1, \mu_{\tilde{\theta}_1}(\tilde{x}_1)) + Q_2^\pi (\tilde{x}_2, \mu_{\theta_2^*}(\tilde{x}_2)) \\&\hspace{40.4mm} \quad  (\tilde{x}_2=f_1(\tilde{x}_1, \mu_{\tilde{\theta}_1}(\tilde{x}_1))) \\ 
&\leq \dots \leq J_0^{\Tilde{\pi}}(x_0)
\end{align*}
where $\Tilde{\pi} = (\Tilde{\theta}_0, \dots, \Tilde{\theta}_{n-1})\in (B(\theta_0^*, \epsilon_0^*)\cap\Theta)\times \dots \times (B(\theta_{n-1}^*, \epsilon_{n-1}^*)\cap\Theta)$.

Choose $\epsilon=\min\{\epsilon_0^*, \dots \epsilon_{n-1}^* \}$. Then, $J_0^\pi(x_0)\leq J_0^{\Tilde{\pi}}(x_0)$ for all $\Tilde{\pi} = (\Tilde{\theta}_0, \dots, \Tilde{\theta}_{n-1})\in (B(\theta_0^*, \epsilon)\cap\Theta)\times \dots \times (B(\theta_{n-1}^*, \epsilon)\cap\Theta)$. This completes the proof.
\end{proof}

\begin{theorem}\label{thm: detparam stationary}
Consider a stationary point of DP $\pi = (\theta_0^*, ..., \theta_{n-1}^*)$. Let the corresponding state sequence be ($x_0^*, \dots x_n^*$). If for every $k \in \{0, \dots, n-1\}$, $\mu_{\theta_k}(x_k)$ is continuously differentiable with respect to $\theta_k$ in a neighborhood of $(x_k^*, \theta_k^*)$, then $\pi$ is also a stationary point of the one-shot optimization.
\end{theorem}

\begin{proof}
    Notice that $\nabla_{\theta_k}J_0^\pi (x_0) = \nabla_{\theta_k}J_k^\pi (x_k^*)=\nabla_{\theta_k}Q_k^\pi (x_k^*, \mu_{\theta_k^*}(x_k^*))$. Thus, $-\nabla_{\theta_k}J_0^\pi (x_0) \in \mathcal{N}_{\Theta}(\theta_k^*)$ for all $k \in \{0,\dots,n-1\}$, which means that $\pi$ is a stationary point of the one-shot optimization.
\end{proof}

\begin{remark}\label{converse detparam stationary}
The converse of Theorem \ref{thm: detparam stationary} clearly does not hold since one can generally find a point $x \in \mathbb{R}^N$ such that $\nabla_{\theta_k}Q_k^\pi (x_k^*, \mu_{\theta_k^*}(x_k^*)) \neq \nabla_{\theta_k}Q_k^\pi (x, \mu_{\theta_k^*}(x))$.     
\end{remark}

\subsection{From one-shot optimization to DP}\label{detparam OSDP}

In this subsection, we first show that a local minimizer of the one-shot optimization does not necessarily correspond to a local minimizer of DP; \textit{i.e.}, the converse of Theorem \ref{thm: detparam loc min} does not hold. Then, with Remark \ref{converse detparam stationary}, it is clear that the optimization landscape of the one-shot optimization is more complex than that of DP. As a by-product, if the one-shot problem has a \textit{low} complexity, so does the DP problem.

To develop a clear counterexample, we restrict the parameterized policy to a certain class as given below, which automatically satisfies the continuity assumption defined in Definition \ref{def: param policy}. 

\begin{definition}\label{def: linear combination}
Define our parameterized policy to be a linear combination of arbitrary linearly independent basis functions, while satisfying Definition \ref{def: param policy}; \textit{i.e.}, Given $m$ functions $f_i : \mathbb{R}^N \rightarrow \mathbb{R}^M$, $i=1,\dots,m$ and $\theta = [s_1,\dots,s_m]^T \in \Theta$, 
\begin{equation}
    \mu_\theta(x)=\sum_{i=1}^{m} s_i f_i(x) \in A,
\end{equation}
where there does not exist $(\Tilde{s}_1,\dots,\Tilde{s}_m)\neq 0$ such that for all $x$ in any set of non-zero measure, the following equation holds \cite{sansone1991function}: 
\begin{equation}\label{independence}
\sum_{i=1}^m \Tilde{s}_i f_i(x) = 0.
\end{equation}
\end{definition}

\begin{remark}\label{discrete continuous}
Since a set of isolated points is a set of measure zero, it is exempt from determining the independence of basis functions. When $x$ has a continuous distribution, the independence of basis functions implies that if (\ref{independence}) holds for all $x$ in the support of the distribution, $\Tilde{s}_1=\dots=\Tilde{s}_m=0$. When $x$ has a discrete distribution, since a set of all the possible values of $x$ is a set of measure zero, the independence of basis functions does not guarantee $\Tilde{s}_1=\dots=\Tilde{s}_m=0$ even if (\ref{independence}) holds for all possible values of $x$.
\end{remark}

Applications of a parameterized policy defined by Definition \ref{def: linear combination} arise in a piecewise polynomial function as well as a stochastic control.
% The details can be found in \cite{supplementary} (see Appendix A). 
The usefulness of the parameterized policy also manifests within Representer theorem \cite{scholkopf2002kernel}: a linear combination of kernels fully represents the solution of minimizing empirical risk. It switches the optimization problem in infinite-dimensional function space to finding the finite number of coefficients. 
% This kernel method can also be extended to multi-dimensional vector outputs \cite{carmeli2006kernel}. 
The minimum number of parameters needed is the number of data points, which is generally much greater than the dimension of the output. Applying this to our parameterized policy, the number of parameters $m$ needs to be greater than the dimension of the action $M$ to cover all data points. For the remainder of this section, we call a policy satisfying $m>M$ as an overparameterized policy.

We now provide some evidence to refute the converse of Theorem \ref{thm: detparam loc min}, specifically if the parameterized policy class is a linear combination of basis functions. It turns out that a local minimizer of the one-shot optimization does not necessarily imply a local minimizer of DP in the overparameterized case.

\begin{proposition}\label{overparam}
Consider an overparameterized policy class defined by Definition \ref{def: linear combination}. Let $\pi=(\theta_0^*, \dots, \theta_{n-1}^*)$ be a local minimizer of DP. If there exists at least one $k\in\{0,\dots,n-1\}$ such that $\theta_k^*$ is in the interior of $\Theta$, then there exists an infinite number of local minimizers of the one-shot optimization corresponding to each local minimizer of DP.
\end{proposition}

\begin{proof}
% The proof is given in \cite{supplementary} (see Appendix B).
Consider the state sequence $(x_0^*, \dots, x_n^*)$ induced by a local minimizer of DP $\pi$. Let $k$ be an index for which $\theta_{k}^*$ is in the interior of $\Theta$. Then, one can express the action taken at step $k$ as $\mu_{\theta_{k}^*}(x_{k}^*)=\sum_{i=1}^{m} s_i^* f_i(x_{k}^*)$ with $\theta_{k}^*=(s_0^*, s_1^*, \dots, s_m^*)$ by Definition \ref{def: linear combination}. Since the policy is overparameterized, $m$ is greater than the dimension of the action $M$. Now, consider the matrix equation
\begin{equation}\label{matrixeq}
\begin{bmatrix}
f_0(x_{k}^*) & f_1(x_{k}^*) & \dots & f_m(x_{k}^*)
\end{bmatrix}
% \begin{bmatrix}
% c_0 \\
% c_1 \\
% \vdots \\
% c_m
% \end{bmatrix}
\theta_{k}
= \mu_{\theta_{k}^*}(x_{k}^*),
\end{equation}
where $\theta_{k}$ is an $m \times 1$ vector variable, and let $F_{k}^*$ denote the first matrix in the left-hand side, which is an $M \times m$ constant matrix given by $x_k^*$. (\ref{matrixeq}) has at least one solution: $\theta_{k}^*$.

The dimension of the null space of $F_{k}^*$ is greater than 0 due to $m>M$. We take any nonzero element $v$ from the null space. Then, for all $\delta\in\mathbb{R}$, $\theta_{k}^*+\delta v$ satisfies (\ref{matrixeq}). Since $\theta_{k}^*$ is in the interior of $\Theta$, one can pick $\epsilon_1 >0$ such that $B(\theta_{k}^*, \epsilon_1 ) \subset \Theta$. Thus, for $0 \leq \delta \leq \frac{\epsilon_1}{\|v\|}$, $\theta_{k}^*+\delta v \in B(\theta_{k}^*, \epsilon_1)$ preserves the state and action sequence associated with $\theta_k^*$ due to (\ref{matrixeq}). The induced cost is also indeed preserved.

By Theorem \ref{thm: detparam loc min}, $\pi$ is a local minimizer of the one-shot optimization. Now, we select $\epsilon_2 >0$ such that $J_0^\pi (x_0) \leq J_0^{\Tilde{\pi}} (x_0)$ for all $\Tilde{\pi}=(\theta_0^*, \dots, \Tilde{\theta}_{k}, \dots, \theta_{n-1}^*)$ where $\Tilde{\theta}_{k} \in B(\theta_{k}^*, \epsilon_2)\cap \Theta$. Let $\epsilon := \min \{\epsilon_1, \epsilon_2\}>0$. 
Then, for $0 \leq \delta \leq \frac{\epsilon}{2\|v\|}$, we have $B(\theta_{k}^*+\delta v, \delta \|v\|) \subset B(\theta_{k}^*, \epsilon)$. Since $\theta_{k}^*+\delta v$ preserves the induced cost, 
% thus $J_0^{\pi'} (x_0)=J_0^\pi (x_0) \leq J_0^{\Tilde{\pi}'} (x_0)$ for all $\Tilde{\pi}'=(\theta_0^*, \dots, \Tilde{\theta}', \dots, \theta_{n-1}^*)$ where $\pi'=(\theta_0^*, \dots, \theta_{k}^*+\delta v, \dots, \theta_{n-1}^*)$ and $\Tilde{\theta}' \in B(\theta_{k}^*+\delta v, \delta \|v\|)\cap \Theta$. Thus, for $0 \leq \delta \leq \frac{\epsilon}{2\|v\|}$, 
$(\theta_0^*, \dots, \theta_{k}^*+\delta v, \dots, \theta_{n-1}^*)$ is a local minimizer of the one-shot optimization for all $0 \leq \delta \leq \frac{\epsilon}{2\|v\|}$. This completes the proof.
\end{proof}

Proposition \ref{overparam} implies that for every $k\in \{0,\dots,n-1\}$, $\theta_k$ of a ``strict" local minimizer of the one-shot optimization does not lie in the interior of $\Theta$. Thus, one can think of constructing a strict local minimizer by restricting the area of $\Theta$. It turns out that given a strict local minimizer of the one-shot optimization and the induced input sequence, no other points can retrieve the same input sequence if $\Theta$ is convex.

\begin{lemma}\label{thetaconvex}
Consider a strict local minimizer of the one-shot optimization $\pi=(\theta_0^*, \dots, \theta_{n-1}^*)$. Let $(x_0^*, \dots, x_n^*)$ be the induced state sequence. Suppose that $\Theta$ is convex and the parameterized policy is defined by Definition \ref{def: linear combination}. Then, $\pi$ is the unique control policy parameter vector that achieves the input sequence $(\mu_{\theta_0^*}(x_0^*), \dots, \mu_{\theta_{n-1}^*}(x_{n-1}^*))$.
\end{lemma}

\begin{proof}
% The proof is given in \cite{supplementary} (see Appendix C).
For every $k \in \{0,\dots,n-1\}$, $\mu_{\theta}(x_k^*)=\sum_{i=1}^{m_k} s_i f_i(x_k^*)$ where $\theta=(s_1,\dots,s_{m_k})$. Let $\theta_k^*=(s_1^*,\dots,s_{m_k}^*)$. Since $\pi$ is a strict local minimizer of the one-shot optimization, $\mu_{\theta}(x_k^*) \neq \mu_{\theta_k^*}(x_k^*)$ in the neighborhood of $\theta_k^*$ if $\theta \neq \theta_k^*$; \textit{i.e.}, there exists $\epsilon>0$ such that 
\begin{align}\label{hyperplane}
\{\theta \in  B(\theta_k^*, \epsilon)\cap\Theta: \sum_{i=1}^{m_k} s_i f_i(x_k^*)=\mu_{\theta_k^*}(x_k^*)\}=\{\theta_k^*\}.
\end{align}
Assume that there exists $\Tilde{\theta} \neq \theta_k^*$ such that $\sum_{i=1}^{m_k} \Tilde{s}_i f_i(x_k^*)=\mu_{\theta_k^*}(x_k^*)$ where $\Tilde{\theta}=(\Tilde{s}_1, \dots, \Tilde{s}_{m_k})$. Then, for $\lambda\in[0,1]$, one can obtain $\sum_{i=1}^{m_k} (\lambda s_i^* +(1-\lambda)\Tilde{s}_i) f_i(x_k^*)=\mu_{\theta_k^*}(x_k^*)$ by linearity and $\lambda \theta_k^* + (1-\lambda)\Tilde{\theta}\in\Theta$ by convexity. Letting $\lambda \rightarrow 1$, one can construct an element of the left-hand side of (\ref{hyperplane}) distinct from $\theta_k^*$. 
By contradiction, $\theta_k^*$ is the unique point that achieves $\mu_{\theta_k^*}(x_k^*)$.
\end{proof}

% \begin{remark}
% If $\Theta$ is not assumed to be convex, the conclusion of Lemma \ref{thetaconvex} does not hold, so the infinite number of control policy parameter vectors can correspond to a single local minimizer of DP as in Proposition \ref{overparam}.    
% \end{remark}

Note that Lemma \ref{thetaconvex} does not necessarily imply that a strict local minimizer of the one-shot optimization is a local minimizer of DP even if $\Theta$ is convex. A simple counterexample can be constructed by considering the 1-step problem 
\begin{align*}
  &c_0(x,\mu_{\theta}(x)) =  \frac{1}{4}\mu_{\theta}(x)^4 -\frac{1}{3} (x^2+2x) \mu_{\theta}(x)^3+ \\ &\hspace{10mm}  \frac{1}{2}(2x^3+x-1)\mu_{\theta}(x)^2 -(x^4-x^3+x^2-x)\mu_{\theta}(x), \\ &c_1(x,\mu_{\theta}(x))= 0, \ f_0(x,\mu_{\theta}(x))=x+\mu_{\theta}(x).
\end{align*}
with the parameterized policy $\mu_{\theta}(x)=d_1 x+d_2$ where $\theta = (d_1,d_2)$ and $\Theta = \{(d_1,d_2): 1 \leq 2d_1-d_2 \leq 3, 1 \leq 2d_1+d_2 \leq 3 \}$ which is convex. At the initial state $x_0=1$, the one-shot problem can be written as
\begin{align*}
    \min_{(d_1,d_2) \in \Theta} \ & \left\{\frac{1}{4}(d_1+d_2)^4 -(d_1+d_2)^3+ (d_1+d_2)^2 \right\}.
\end{align*}

Each vector $(d_1, d_2) \in \Theta $ which satisfies $d_1+d_2=0$ or $d_1+d_2=2$ is a local minimizer of the one-shot optimization. 
Since $\{(d_1,d_2) :d_1+d_2=0\}\cap \Theta = \{(1,-1)\}$ and $\{(d_1,d_2) :d_1+d_2=2\}\cap \Theta=\{(1,1)\}$, we have $(1,-1)$ and $(1,1)$ as strict local minimizers of the one-shot optimization. On the other hand, since $\nabla_{\theta} Q_0^\pi(x,\mu_{\theta}(x))=\nabla_{\theta} c_0(x,\mu_{\theta}(x))=[g(x,\theta)x, g(x,\theta)]^T$ where $g(x,\theta)=(\mu_{\theta}(x)-(x^2+1))(\mu_{\theta}(x)-x)(\mu_{\theta}(x)-(x-1))$, a local minimizer of DP should be the parameter that yields $\mu_{\theta}(x)=x-1$ or $\mu_{\theta}(x)=x^2+1$ for all $x\in \mathbb{R}^N$. Since a linear policy cannot contain $x^2+1$, $(1,-1) \in \Theta$ is the only local minimizer of DP. Thus, $(1,1)$ is a strict local minimizer of the one-shot optimization but not a local minimizer of DP. Fig.~\ref{fig: deterministic one-shot not DP} shows the domain and the landscape of the one-shot optimization. 

\begin{figure}[t]
\centering 
\subfloat[Domain]{\label{fig: deterministic 2D}\includegraphics[width=24mm]{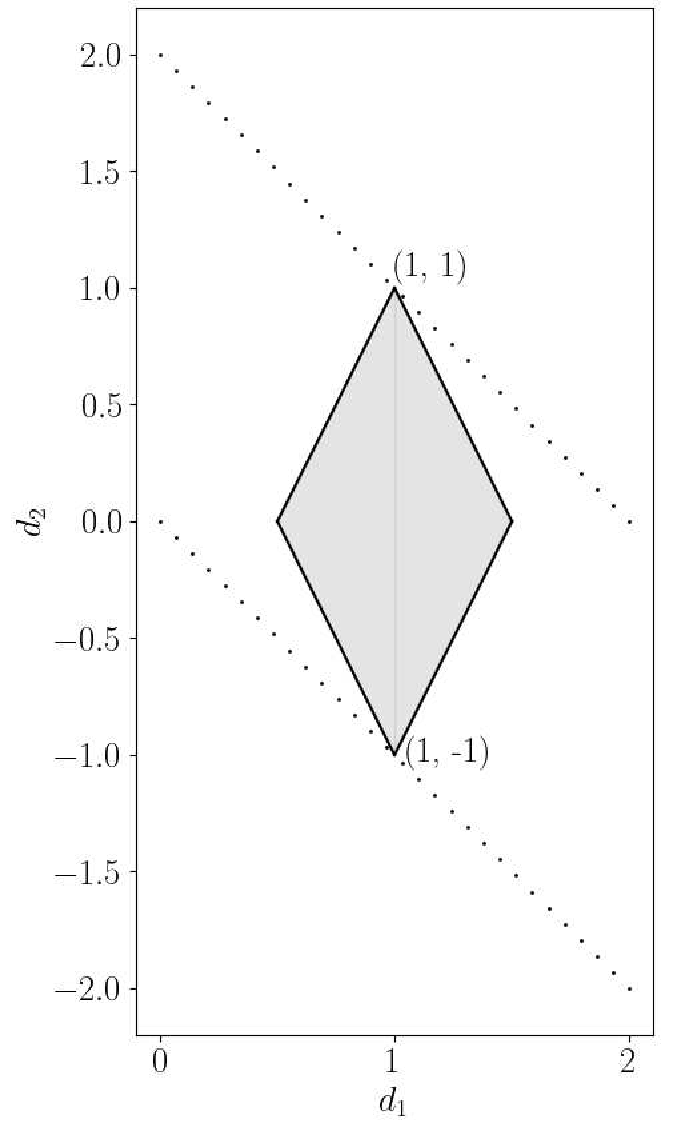}} \hspace{1mm} \subfloat[Landscape]{\label{fig:deterministic 3D}\includegraphics[width=48mm]{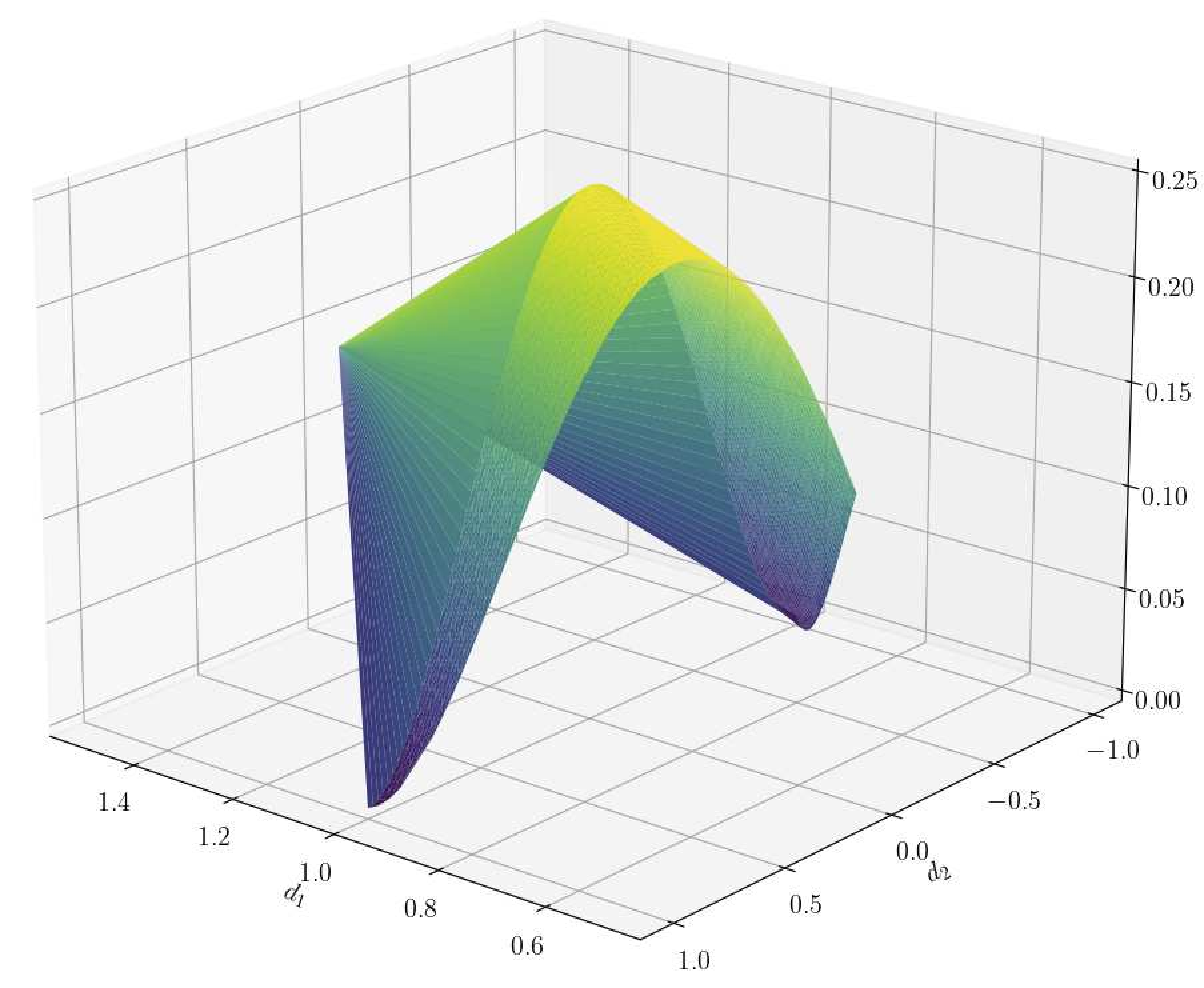}}
\caption{The domain and the landscape of the one-shot optimization for a deterministic parameterized problem: (a) The gray-colored area is the domain of the parameter space. The intersection between the dotted lines and the domain is $\{(1,1), (1,-1)\}$. (b) Both $(1,1)$ and $(1,-1)$ are a strict local minimizer of the one-shot optimization but only $(1,-1)$ is a local minimizer of DP.}
\label{fig: deterministic one-shot not DP}
% \vspace{-3mm}
\end{figure}

In light of the above counterexample, one can think of the situation where the parameterized policy contains every locally minimum control policy of DP (see Definition \ref{def: local min DP}). It turns out that if such a situation is possible, given a convex parameter space, each strict local minimizer of the one-shot optimization is a local minimizer of DP under the following assumptions.

\begin{assumption}\label{full-row rank}
    Given a local minimizer of the one-shot optimization $\pi$, let $(x_0^*,\dots,x_n^*)$ be the associated state sequence. Then, for all $k\in\{0,\dots,n-1\}$, the $M \times m$ matrix $[f_0(x_{k}^*) \ f_1(x_{k}^*) \ \dots \ f_m(x_{k}^*)]$ has a full row rank.
\end{assumption}

\begin{assumption}\label{det theta large}
    Assume that $A \subseteq \cap_{k=1}^n \mu_{\Theta}(x_k^*)$, where $\mu_{\Theta}(x_k^*)$ is the image of $\Theta$ through $\mu_{\theta}(x_k^*): \Theta \rightarrow A$.
\end{assumption}

\begin{lemma}\label{lemma:every loc min policy contain}
Assume that $\Theta$ is convex. Consider a strict local minimizer of the one-shot optimization $\pi=(\theta_0^*,\dots,\theta_{n-1}^*)$. Suppose that the parameterized policy defined by Definition \ref{def: linear combination} satisfies Assumptions \ref{full-row rank} and \ref{det theta large}. If the parameterized policy class contains every locally minimum control policy of DP and at least one of the locally minimum control policies satisfies $\inf_{x\in \mathbb{R}^N} \epsilon_k^* (x) >0$ for all $k\in \{0,\dots,n-1\}$, then $\pi$ is a local minimizer of DP.
\end{lemma}

\begin{proof}
Let $(x_0^*,\dots,x_n^*)$ be the state sequence associated with $\pi$. Recall that $J_0^{\pi}(x_0)=\sum_{i=0}^{k-1}c_i(x_i^*, \mu_{\theta_i^*}(x_i^*)) + Q_k^\pi (x_k^*, \mu_{\theta_k^*}(x_k^*))$. One can fix all parameters except $\theta_k^*$ to derive that $J_0^{\pi}(x_0) -  J_0^{\pi'}(x_0)=Q_k^{\pi}(x,\mu_{\theta_k^*}(x)) - Q_k^{\pi}(x,\mu_{\theta_k'}(x))$, where $\pi'=(\theta_0^*, \dots, \theta_{k-1}^*, \theta_k', \theta_{k+1}^*, \dots, \theta_{n-1}^*)$. Thus, a local minimizer of the one-shot optimization $\pi$ implies that for all $k\in\{0,\dots,n-1\}$, there exists $\epsilon_k^*>0$ such that 
\begin{equation}\label{one-shot to Q}
    Q_k^{\pi}(x_k^*,\mu_{\theta_k^*}(x_k^*)) \leq Q_k^{\pi}(x_k^*,\mu_{\Tilde{\theta}}(x_k^*)), \quad\forall \Tilde{\theta} \in B(\theta_k^*,\epsilon_k^\ast)\cap \Theta.
\end{equation}

Now, let $F_k^*$ be the $M \times m$ matrix $[f_0(x_{k}^*) \ f_1(x_{k}^*) \ \dots \ f_m(x_{k}^*)]$, where its smallest singular value is denoted by $\sigma_k^*$. Given an arbitrary direction $v\in \mathbb{R}^M$, one can take a point $u_v$ that is farthest from $\mu_{\theta_k^*}(x_k^*)$ in the direction of $v$ since the action space $A$ is compact. Let $\delta_v$ be the value that achieves $u_v=\mu_{\theta_k^*}(x_k^*)+\delta_v v$. By Assumption \ref{det theta large}, there exists $\theta_v \in \Theta$ satisfying $u_v=\mu_{\theta_v}(x_k^*)$, and by Definition \ref{def: linear combination}, $\mu_{\theta_v}(x_k^*)$ is defined by $F_k^* \theta_v$.

% Note that $A$ is necessarily convex. To verify this, take $a_1, a_2 \in A$ and one can retrieve the corresponding $\theta_1, \theta_2 \in \Theta$ that satisfies $a_1 = \mu_{\theta_1}(x_k^*)=F_k^* \theta_1$ and $a_2 = \mu_{\theta_2}(x_k^*)=F_k^* \theta_2$. Since $\Theta$ is convex, $(1-\lambda)\theta_1 + \lambda \theta_2 \in \Theta$ for $\lambda \in [0,1]$. By linearity, $(1-\lambda)a_1 + \lambda a_2 =\mu_{(1-\lambda)\theta_1 + \lambda \theta_2}(x_k^*)$, so  

\noindent\underline{\textit{Case 1}} $\delta_v = 0$: There does not exist $\delta>0$ such that $\mu_{\theta_k^*}(x_k^*)+\delta v \in A$.

\noindent\underline{\textit{Case 2}} $\delta_v > 0$ and $\theta_v \in B(\theta_k^*, \epsilon_k^*)$: Due to the linearity of policy and the convexity of $\Theta$, there exists $\theta_{\delta} \in B(\theta_k^*, \epsilon_k^*) \cap \Theta$ such that $\mu_{\theta_{\delta}}(x_k^*)=\mu_{\theta_k^*}(x_k^*)+\delta v$ for all $0<\delta<\delta_v$.

\noindent\underline{\textit{Case 3}} $\delta_v > 0$ and $\theta_v \notin B(\theta_k^*, \epsilon_k^*)$: Consider $\mu_{\theta_k^*}(x_k^*)+ \frac{\epsilon_k^*}{2\|\theta_v -\theta_k^*\|} (\mu_{\theta_v}(x_k^*)-\mu_{\theta_k^*}(x_k^*))$. The corresponding parameter is definitely in $B(\theta_k^*, \epsilon_k^*)\cap\Theta$ by the linearity of policy and the convexity of $\Theta$. Then, as in Case 2, there exists $\theta_{\delta} \in B(\theta_k^*, \epsilon_k^*) \cap \Theta$ such that $\mu_{\theta_{\delta}}(x_k^*)=\mu_{\theta_k^*}(x_k^*)+\delta v$ for all $0<\delta<\frac{\epsilon_k^*}{2\|\theta_v -\theta_k^*\|} \delta_v$.

In Case 3, notice that $\|\frac{\epsilon_k^*}{2\|\theta_v -\theta_k^*\|} (\mu_{\theta_v}(x_k^*)-\mu_{\theta_k^*}(x_k^*)) \| =\frac{\epsilon_k^*}{2} \cdot \frac{\|F_k^* (\theta_v - \theta_k^*)\|}{\|\theta_v -\theta_k^*\|} \geq \frac{\epsilon_k^*}{2} \sigma_k^* > 0$, where the last inequality is from Assumption \ref{full-row rank} and the second last inequality is from the basic property of singular value \cite{strang1988linear}. 

Considering all three cases, $\Tilde{u} \in B(\mu_{\theta_k^*}(x_k^*), \frac{\epsilon_k^*}{2} \sigma_k^*) \cap A$ implies that at least one corresponding parameter for each $\Tilde{u}$ is in $B(\theta_k^*, \epsilon_k^*) \cap \Theta$. Thus, one can notice that (\ref{one-shot to Q}) implies 
\begin{equation}\label{DP to Q}
    Q_k^{\pi}(x_k^*,\mu_{\theta_k^*}(x_k^*)) \leq Q_k^{\pi}(x_k^*, \Tilde{u}), \quad\forall \Tilde{u} \in B(\mu_{\theta_k^*}(x_k^*), \frac{\epsilon_k^*}{2} \sigma_k^*) \cap A.
\end{equation}

We select an arbitrary locally minimum control policy $\phi=(\phi_{0},\dots, \phi_{n-1})$ with the property that $\inf_{x\in \mathbb{R}^N} \epsilon_k^* (x) >0$. Let $\Tilde{\pi}=(\Tilde{\pi}_{0},\dots, \Tilde{\pi}_{n-1})$ be the policy such that for all $k\in\{0,\dots,n-1\}$,
\[
\Tilde{\pi}_{k}(x_k)=\begin{dcases*}
\mu_{\theta_k^*}(x_k^*), & if $x_k =x_k^*$, \\
\phi_{k} (x_k), & otherwise.
\end{dcases*}
\]

Such $\Tilde{\pi}$ is also a locally minimum control policy by (\ref{DP to Q}). This implies that the parameterized policy contains $\Tilde{\pi}$. Also, $\Tilde{\pi}$ achieves the same input sequence $(\mu_{\theta_0^*}(x_0^*), \dots, \mu_{\theta_{n-1}^*}(x_{n-1}^*))$ as the strict local minimizer $\pi$. Therefore, by Lemma \ref{thetaconvex}, $\mu_{\theta_k^*}=\Tilde{\pi}_{k}$ holds. Since $\inf_{x\in\mathbb{R}^N} \epsilon_k^* (x)$ induced by $\phi_{k}$ is greater than 0, $\inf_{x\in\mathbb{R}^N} \epsilon_k^* (x)$ induced by $\Tilde{\pi}_{k}$ is also greater than 0. Then, by Proposition \ref{prop: action and parameter}, $\pi=(\theta_0^*,\dots,\theta_{n-1}^*)$ is a local minimizer of DP.
\end{proof}

\begin{remark}
    With a given set of parameters $(\theta_0^*, \dots, \theta_{n-1}^*)$, there exists only one associated state sequence for the deterministic parameterized problem. Assumptions \ref{full-row rank} and \ref{det theta large} are thus only required for that specific state sequence, where one can readily check the assumptions in advance with known dynamics, parameter space, action space, and policy class.
    Assumption \ref{full-row rank} is a type of regularity condition, which can be regarded as the extension of an overparameterized policy.
    Assumption \ref{det theta large} implies that $\Theta$ should be large enough to contain relevant parameters to cover the action space $A$.
    Since $\mu_\theta(x)$ is designed to be in $A$ by Definition \ref{def: param policy}, Assumption \ref{det theta large} is equivalent to saying that $A=\mu_{\Theta}(x_0^*)=\dots=\mu_{\Theta}(x_n^*)$. 

    % Assumption \ref{det theta large} implies that $\Theta$ should be large enough to contain relevant parameters to cover the action space $A$. Also, Assumption \ref{full-row rank} can be regarded as the extension of an overparameterized policy.
    
    % On the other hand, if $\Theta$ is not large enough, $\Tilde{\pi}$ may not be a locally minimum control policy in that a neighborhood of $\theta_k^* \in \Theta$ may not cover a neighborhood of $\mu_{\theta_k^*}(x_k^*) \in A$. One can think of an extreme situation that $\Theta$ is a line but the dimension of $A$ is greater than 1.
\end{remark}

Meanwhile, suppose that there exist two different locally minimum control policies in a set of non-zero measure, meaning that at some step $k$, $\pi_1(x)\neq \pi_2(x)$ for all $x\in I$ where $I$ is a set of non-zero measure. Then, there exists an infinite number of locally minimum control policies made up of $\pi_1$ and $\pi_2$ by alternating between $\pi_1(x)$ and $\pi_2(x)$ along $x\in I$, and the parameterized policy class cannot contain all these policies. We now present the situation that the parameterized policy contains every locally minimum control policy of DP. 

\begin{theorem}\label{thm: detparam one-shot to DP}

Assume that $\Theta$ is convex. Consider a strict local minimizer of the one-shot optimization $\pi=(\theta_0^*,\dots,\theta_{n-1}^*)$. Suppose that the parameterized policy defined by Definition \ref{def: linear combination} satisfies Assumptions \ref{det theta large} and \ref{full-row rank}. If there exists only a single locally minimum control policy of DP $\phi=(\phi_{0},\dots, \phi_{n-1})$ and the parameterized policy class contains $\phi$, then $\pi$ is a local minimizer of DP.
\end{theorem}

\begin{proof}
    Let $\phi' = (\theta_0', \dots, \theta_{n-1}')$ be the parameters associated with $\phi$. For all $k\in\{0,\dots,n-1\}$ and for all $x\in\mathbb{R}^N$, $\phi_{k}(x)$ is the unique local minimizer of $Q_k^{\phi'}(x,u)$. Having no spurious local minima implies that $\inf_{x\in\mathbb{R}^N} \epsilon_k^* (x) =\infty >0$. Moreover, the parameterized policy class contains every locally minimum control policy of DP. Since these facts satisfy the preconditions of Lemma \ref{lemma:every loc min policy contain}, this completes the proof.
\end{proof}

Considering both Theorem \ref{thm: detparam loc min} and \ref{thm: detparam one-shot to DP}, one can conclude that under the assumptions of Theorem \ref{thm: detparam one-shot to DP}, a local minimizer of DP is \textit{equivalent} to a local minimizer of the one-shot optimization.
% In other words, if DP has a very low complexity, so does the one-shot problem.

\section{Stochastic Problem under a \\ Parameterized policy}
\label{sec:stoparam}

\subsection{Problem Formulation}\label{prob-form stoparam}

In this section, we will show that the results obtained for the deterministic problem under a parameterized policy also hold for the stochastic problem under a parameterized policy. Since we now take the expectation of the sum of the costs over the trajectories, the issue of strictness, as in Proposition \ref{overparam}, does not take place. Before presenting the theorems, we first define the problem setting in the stochastic case.

\begin{definition}
    Given a complete probability space $(\Omega, \mathcal{F}, \mathbb{P})$, let $x_0$ be a $\mathcal{F}$-measurable, $\mathbb{R}^N$-valued random variable, which has an initial distribution $\rho$. Also, let $w_k$ be an $\mathcal{F}$-measurable, $\mathbb{R}^W$-valued random variable for all $k\in\{0,\dots,n-1\}$ such that $x_0, w_0,\dots,w_{n-1}$ are mutually independent. The state transition is now governed by the dynamics $f_i : \mathbb{R}^N \times A \times \mathbb{R}^W \rightarrow \mathbb{R}^N$, $i=0,\dots,n-1$. The dynamics are again defined to be at least twice continuously differentiable.
\end{definition}

Now, we modify the deterministic problems under a parameterized policy, \textit{i.e.}, (\ref{eq:DP1}), (\ref{eq:DP2}), and (\ref{eq:DDP}), to a discrete-time finite-horizon stochastic optimal control problem under a parameterized policy:
\begin{equation}\label{eq: SP1}
\tag{SP1}
\begin{alignedat}{2}
\min_{\theta_0,\dots,\theta_{n-1}\in \Theta} \quad & \mathbb{E}_{x_0,w_0,\dots,w_{n-1}}\biggr[\sum_{i=0}^{n-1} c_i(x_i,\mu_{\theta_i}(x_i))+ c_n(x_n)\biggr], \\
\nonumber \text{where} \quad & x_{i+1}=f_i(x_i,\mu_{\theta_i}(x_i), w_i), \quad i=0,\dots, n-1.
\end{alignedat}
\end{equation}
 Notice that for stochastic problems, $x_0$ may not be given as a point, but has an initial distribution $\rho$. Afterwards, $x_{i+1}$ is a random variable induced by $(x_0, w_0,\dots,w_{i})$.
\begin{definition} \label{def: stoparam Q and cost-to-go}
Given a control policy parameter vector $\pi=(\theta_0,\dots,\theta_{n-1})$,
the associated Q-functions $Q^\pi_k(\cdot,\cdot)$ and cost-to-go functions $J^\pi_k(\cdot)$ under the policy $\pi$ are defined in a backward way from the time step $n-1$ to the time step $0$ through the following recursion:
% \abovedisplayskip
\vspace{-3pt}
\begin{align*}
&J^\pi_n(x)=c_n(x),\\
&Q^\pi_k(x, \mu_{\theta}(x))=\mathbb{E}_{w_k}[c_k(x,\mu_{\theta}(x))+J^\pi_{k+1}(f_k(x,\mu_{\theta}(x), w_k))],\\ &\hspace{63mm} k=0,\dots,n-1, \\
&J^\pi_k(x)=Q^\pi_k(x,\mu_{\theta_k}(x)),\quad  k=0,\dots,n-1.
\end{align*}
\end{definition}
\vspace{3pt}
Then, the one-shot optimization problem \eqref{eq: SP1} can be equivalently written as
\begin{equation}\label{eq:SP2}
\tag{SP2}
\begin{split}
\min_{\pi=(\theta_0,\dots,\theta_{n-1}) \in \Theta^n} \quad & \mathbb{E}_{x_0}[J_0^\pi (x_0)],
\end{split}
\end{equation}
as long as the cost functions $c_i$, $i=0,\dots,n-1$, are uniformly bounded, due to the product measure Theorem and Fubini's Theorem \cite{bertsekas1996stochastic}. In the remainder of the paper, we assume that the two problems are equivalent.

The DP approach can be written as the following backward recursion:
\begin{equation}\label{eq:SDP}
\tag{SP3}
\begin{split}
J_n(x)&=c_n(x),\\
J_k(x)&=\min_{\theta \in \Theta}\{\mathbb{E}_{w_k}[c_k(x,\mu_{\theta}(x))+J_{k+1}(f_k(x,\mu_{\theta}(x),w_k))]\},\\ &\hspace{55mm} k=0,\dots,n-1.
\end{split}
\end{equation}

\begin{definition}[local minimizer of the one-shot optimization]\label{def: stoparam local min OS}
A control policy parameter vector $\pi=(\theta_0^*,\dots,\theta_{n-1}^*)$ is said to be a local minimizer of the one-shot optimization if there exists $\epsilon>0$ such that
\[ \mathbb{E}_{x_0}[J_0^{\pi}(x_0)] \leq  \mathbb{E}_{x_0}[J_0^{\Tilde{\pi}}(x_0)]\]
for all $\Tilde{\pi}=(\Tilde{\theta}_0, \dots, \Tilde{\theta}_{n-1}) \in (B(\theta_0^*,\epsilon)\cap \Theta) \times \dots \times (B(\theta_{n-1}^*,\epsilon)\cap \Theta)$. 
% It is further called a spurious local minimizer of the one-shot optimization problem if $\mathbb{E}_{x_0}[J_{0}^{\pi}(x_0)]>\mathbb{E}_{x_0}[J_0(x_0)]$.
\end{definition}

% \begin{definition}[local minimizer of DP] \label{def: stoparam local min DP}
% A control policy parameter vector $\pi=(\theta_0^*,\dots,\theta_{n-1}^*)$
% is said to be a local minimizer of DP if for all $k\in \{0,\dots,n-1\}$ and for all $x \in \mathbb{R}^N$, the policy parameter $\theta_k^*$ is a local minimizer of the Q-function $Q^\pi_k(x,\mu_{(\cdot)} (x))$, meaning that there exists $\epsilon_k^\ast>0$ such that 
% \begin{equation*}
%     Q_k^{\pi}(x,\mu_{\theta_k^*}(x)) \leq Q_k^{\pi}(x,\mu_{\Tilde{\theta}}(x)) ,\ \forall \Tilde{\theta} \in B(\theta_k^*,\epsilon_k^\ast)\cap \Theta.
% \end{equation*}
% % It is further called a spurious local minimizer of DP if $\mathbb{E}_{x_0}[J_{0}^{\pi}(x_0)]>\mathbb{E}_{x_0}[J_0(x_0)]$.
% \end{definition}

\begin{definition}[Stationary point of the one-shot optimization]\label{def: stoparam stationary point of the one-shot optimization}
A control policy parameter vector $\pi=(\theta_0^*,\dots,\theta_{n-1}^*)$ is said to be a stationary point of the one-shot optimization if for all $k \in \{0,\dots,n-1\}$, it holds that $-\nabla_{\theta_k}\mathbb{E}_{x_0}[J_0^{\pi}(x_0)] \in \mathcal{N}_{\Theta}(\theta_k^*)$.
\end{definition}

While the one-shot method aims for optimizing the expectation over all steps in the stochastic dynamics, DP studies optimizing Q-function at every step both in the deterministic and stochastic cases. Since we have modified the definition of Q-function to incorporate the expectation, it is natural that the definition of a local minimizer (stationary point) of DP is exactly the same as Definition \ref{def: detparam local min DP} (\ref{def: detparam stationary point of DP}). 

\subsection{From DP to one-shot optimization}\label{stoparam DPOS}

In this subsection, we will show that, in the stochastic case with a parameterized policy, each local minimizer (stationary point) of DP directly corresponds to some local minimizer (stationary point) of the one-shot optimization, just as in the deterministic case. However, it turns out that for the stationary points, the policy needs to be continuously differentiable with respect to both states and parameters since the expectation is over all trajectories rather than a single trajectory.

\begin{theorem}\label{thm: stoparam loc min}
Consider a local minimizer of DP $\pi = (\theta_0^*, ..., \theta_{n-1}^*)$. Then, $\pi$ is also a local minimizer of the one-shot optimization.
\end{theorem}

\begin{proof}

Since $(\theta_0^*, ..., \theta_{n-1}^*)$ is a local minimizer of DP, there exist $\epsilon_0^*, \dots, \epsilon_{n-1}^* > 0$ such that 
\begin{align*}
\mathbb{E}_{x_0}[J_0^\pi(x_0)]&=\mathbb{E}_{x_0}[Q_0^\pi (x_0, \mu_{\theta_0^*}(x_0))] \leq \mathbb{E}_{x_0}[Q_0^\pi (x_0, \mu_{\tilde{\theta}_0}(x_0))] \\ 
&= \mathbb{E}_{x_0}[c_0(x_0, \mu_{\tilde{\theta}_0}(x_0)) + \mathbb{E}_{w_0}[Q_1^\pi (\tilde{x}_1, \mu_{\theta_1^*}(\tilde{x}_1))]] \\ &\hspace{34mm} (\tilde{x}_1=f_0(x_0, \mu_{\tilde{\theta}_0}(x_0),w_0)) \\ 
&\leq \mathbb{E}_{x_0}[c_0(x_0, \mu_{\tilde{\theta}_0}(x_0)) + \mathbb{E}_{w_0}[Q_1^\pi (\tilde{x}_1, \mu_{\tilde{\theta}_1}(\tilde{x}_1))]] \\ 
&= \mathbb{E}_{x_0}[c_0(x_0, \mu_{\tilde{\theta}_0}(x_0)) +  \mathbb{E}_{w_0}[c_1(\tilde{x}_1, \mu_{\tilde{\theta}_1}(\tilde{x}_1)) \\ & \hspace{33.5mm} +\mathbb{E}_{w_1}[Q_2^\pi (\tilde{x}_2, \mu_{\theta_2^*}(\tilde{x}_2))]]] \\ &\hspace{34mm} (\tilde{x}_2=f_1(\tilde{x}_1, \mu_{\tilde{\theta}_1}(\tilde{x}_1),w_1)) \\ 
&\leq \dots \leq \mathbb{E}_{x_0,w_0,\dots,w_{n-1}}[J_0^{\Tilde{\pi}}(x_0)]
\end{align*}
where $\Tilde{\pi} = (\Tilde{\theta}_0, \dots, \Tilde{\theta}_{n-1})\in (B(\theta_0^*, \epsilon_0^*)\cap\Theta)\times \dots \times (B(\theta_{n-1}^*, \epsilon_{n-1}^*)\cap\Theta)$. The last inequality is due to the assumption that the two problems (\ref{eq: SP1}) and (\ref{eq:SP2}) are equivalent.

Choose $\epsilon=\min\{\epsilon_0^*, \dots \epsilon_{n-1}^* \}$. Then, $J_0^\pi(x_0)\leq J_0^{\Tilde{\pi}}(x_0)$ for all $\Tilde{\pi} = (\Tilde{\theta}_0, \dots, \Tilde{\theta}_{n-1})\in (B(\theta_0^*, \epsilon)\cap\Theta)\times \dots \times (B(\theta_{n-1}^*, \epsilon)\cap\Theta)$. This completes the proof.
\end{proof}

Now, let $\mathbf{D}_x^{\mu}(\theta)$ be the Jacobian matrix of $\mu_{(\cdot)}(x)$ at point $\theta$, $\mathbf{D}_{k}^{f,x} (x, \mu_{\theta}(x), w)$ be the Jacobian matrix of the function $f_k(\cdot, \mu_{\theta}(\cdot), w)$ at point $x$ while viewing $\theta$ as a constant, and similarly $\mathbf{D}_k^{f,\theta} (x, \mu_{\theta}(x), w)$ be the Jacobian matrix of $f_k(x, \mu_{(\cdot)}(x), w)$ at point $\theta$ while viewing $x$ as a constant.

% \clearpage

\begin{theorem} \label{thm: stoparam stationary}
Consider a stationary point of DP $\pi=(\theta_0^*,\dots,\theta_{n-1}^*)$.
%meaning that $-\nabla_{\theta_k} Q_k^\pi (x, \mu_{\theta_k^*}(x)) \in \mathcal{N}_{\Theta}(\theta_k^*)$ for all $k \in \{0,\dots,n-1\}$ and for all $x \in \mathbb{R}^N$. 
If for all $k\in\{0,\dots,n-1\}$,
\begin{enumerate}
\item $\mu_{\theta_k}(x_k^*)$ is continuously differentiable with respect to $\theta_k$ in a neighborhood of $\theta^*_k$ for all $x_k^* \in \mathbb{R}^N$;
\item $\mu_{\theta_k^*}(x_k)$ is continuously differentiable with respect to $x_k$ everywhere,
\end{enumerate}
then $\pi$ is a stationary point of the one-shot optimization.
\end{theorem}
\begin{proof}

First, we will apply induction to prove that for every $k \in \{1, ..., n\}, J_k^\pi (x)$ is continuously differentiable. For the base step, $J_n^\pi(x)=c_n(x)$ is continuously differentiable. For the induction step, observe that
\begin{align*}
&\nabla_{x}J_k^\pi (x) = \nabla_{x} [Q_k^\pi (x, \mu_{\theta_k^*}(x))] \\&= \nabla_{x} [c_k (x, \mu_{\theta_k^*}(x)) + \int_{\Omega} J_{k+1}^\pi (f_k (x, \mu_{\theta_k^*}(x), w_k)) \mathrm{d}p(w_k)] \\
 &= \nabla_{x} [c_k (x, \mu_{\theta_k^*}(x))] +  \int_{\Omega} \nabla_{x} [J_{k+1}^\pi (f_k (x, \mu_{\theta_k^*}(x), w_k)) ]\mathrm{d}p(w_k) \\
 &= \nabla_{x} [c_k (x, \mu_{\theta_k^*}(x))] \\ &\hspace{1mm}+  \int_{\Omega} \mathbf{D}_{k}^{f,x} (x, \mu_{\theta_k^*}(x), w_k)^T \nabla_{x} J_{k+1}^\pi (f_k (x, \mu_{\theta_k^*}(x), w_k)) \mathrm{d}p(w_k).
\end{align*}
This observation is based on the existence and continuity of the Jacobian matrix $\mathbf{D}_{k}^{f,x} (x, \mu_{\theta_k^*}(x), w_k)$ due to assumption 2, continuity of $\nabla_{x} J_{k+1}^\pi (f_k (x, \mu_{\theta_k^*}(x), w_k))$ due to the induction step, and therefore the continuity of $\nabla_{x} [J_{k+1}^\pi (f_k (x, \mu_{\theta_k^*}(x), w_k)) ]$. This allows us to interchange integration and differentiation in the second equality by Leibniz's integration rule.

Now, for $k \in \{0, ..., n-1\}$, observe that 
\begin{align*}
& \nabla_{\theta_k} Q_k^\pi (x_k, \mu_{\theta_k^*}(x_k)) \\&= \nabla_{\theta_k} [c(x_k, \mu_{\theta_k^*}(x_k))+\int_{\Omega} J_{k+1}^\pi (f_k (x_k, \mu_{\theta_k^*}(x_k), w_k))\mathrm{d}p(w_k)] \\
&=\mathbf{D}_{x_k}^{\mu}(\theta_k^*)^T \nabla_{\mu} c(x_k, \mu_{\theta_k^*}(x_k)) \\&+ \int_{\Omega} \mathbf{D}_k^{f,\theta} (x_k, \mu_{\theta_k^*}(x_k), w_k)^T \nabla_x J_{k+1}^\pi (f_k (x_k, \mu_{\theta_k^*}(x_k), w_k))  \mathrm{d}p(w_k),
\end{align*}
which is valid because for $k \in \{1,..., n\}, J_k^\pi (x)$ is continuously differentiable and assumption 1 implies the existence and continuity of $\mathbf{D}_{x_k}^{\mu}(\theta_k^*)$ and $\mathbf{D}_k^{f,\theta} (x_k, \mu_{\theta_k^*}(x_k), w_k)$. Thus, $\nabla_{\theta_k} Q_k^\pi (x_k, \mu_{\theta_k}(x_k))$ is continuous in a neighborhood of $\theta_k^*$ for all $x_k\in\mathbb{R}^N$. Then, for $k \in \{0, ..., n-1\}$, 
\begin{align*}
&\nabla_{\theta_k} \mathbb{E}_{x_0}[J_0^{\pi}(x_0)] = \\ &\int_{\mathbb{R}^N}\int_{\Omega}\dots\int_{\Omega} \nabla_{\theta_k} Q_k^\pi (x_k, \mu_{\theta_k^*}(x_k))\mathrm{d}p(w_{k-1})...\mathrm{d}p(w_{0})\mathrm{d}p(x_{0}), 
\end{align*}
% which is due to the continuity of $\nabla_{\theta_k} Q_k^\pi (x_k, \mu_{\theta_k}(x_k))$ around $\theta_k^*$ for all $x_k\in\mathbb{R}^N$.
Now, note that $\mathcal{N}_{\Theta}(\theta_k^*)$ is nonempty, closed, and convex \cite{rockafellar2009variational}. By the definition of a stationary point of DP, we have $-\nabla_{\theta_k} Q_k^\pi (x_k, \mu_{\theta_k^*}(x_k)) \in \mathcal{N}_{\Theta}(\theta_k^*)$ for all $x_k \in \mathbb{R}^N$. To prove by contradiction, assume that $-\nabla_{\theta_k} \mathbb{E}_{x_0}[J_0^{\pi}(x_0)] \notin \mathcal{N}_{\Theta}(\theta_k^*)$. Let $a_k$ denote the dimension of $\theta_k$. By the separating hyperplane theorem, there exist $p \in \mathbb{R}^{a_k}$ and $\alpha \in \mathbb{R}$ such that 
\[-p^T \nabla_{\theta_k} Q_k^\pi (x_k, \mu_{\theta_k^*}(x_k)) < \alpha < -p^T \nabla_{\theta_k} \mathbb{E}_{x_0}[J_0^{\pi}(x_0)], 
\]
for all $x_k \in \mathbb{R}^N.$ Then, observe that
\begin{align*}
&-p^T \nabla_{\theta_k} \mathbb{E}_{x_0}[J_0^{\pi}(x_0)] \\ &= -p^T \int_{\mathbb{R}^N}\int_{\Omega}\dots\int_{\Omega} \nabla_{\theta_k} Q_k^\pi (x_k, \mu_{\theta_k^*}(x_k))\mathrm{d}p(w_{k-1})...
% \mathrm{d}p(w_{0})
\mathrm{d}p(x_{0}) \\
&= \int_{\mathbb{R}^N}\int_{\Omega}\dots\int_{\Omega} -p^T \nabla_{\theta_k} Q_k^\pi (x_k, \mu_{\theta_k^*}(x_k))\mathrm{d}p(w_{k-1})...
% \mathrm{d}p(w_{0})
\mathrm{d}p(x_{0}) \\
&< \int_{\mathbb{R}^N}\int_{\Omega}\dots\int_{\Omega} -p^T \nabla_{\theta_k} \mathbb{E}_{x_0}[J_0^{\pi}(x_0)]\mathrm{d}p(w_{k-1})...
% \mathrm{d}p(w_{0})
\mathrm{d}p(x_{0}) \\ &= -p^T \nabla_{\theta_k} \mathbb{E}_{x_0}[J_0^{\pi}(x_0)],
\end{align*}
which is a contradiction. Thus, $-\nabla_{\theta_k} \mathbb{E}_{x_0}[J_0^{\pi}(x_0)] \in \mathcal{N}_{\Theta}(\theta_k^*)$,
which shows that $\pi=(\theta_0^*, ..., \theta_{n-1}^*)$ is a stationary point of the one-shot optimization. 
\end{proof}

% \begin{remark}
%     Notice that the parameterized policy class defined by Definition \ref{def: linear combination} satisfies the condition of Theorem \ref{thm: stoparam stationary} as long as all basis functions are continuously differentiable. Therefore, one can apply the result to a linear combination of arbitrary continuously differentiable functions.
% \end{remark}

\subsection{From one-shot optimization to DP}\label{stoparam OSDP}

In this subsection, we first show that a local minimizer (stationary point) of the one-shot optimization does not necessarily correspond to a local minimizer (stationary point) of DP; \textit{i.e.}, the converse of Theorem \ref{thm: stoparam loc min} and that of Theorem \ref{thm: stoparam stationary} do not hold. 
Then, the optimization landscape of the one-shot optimization is more complex than that of DP. In other words, if the one-shot problem has a \textit{low} complexity, so does the DP problem.

To provide a counterexample, we use the basic parameterized policy that follows Definition \ref{def: linear combination}: $\mu_{\theta_k}(x)=a_k x+b_k$, where $\theta_k = (a_k,b_k)$. Consider the 2-step problem
\begin{align*}
&x_0 = 0, \ c_0(x, \mu_{\theta_0}(x)) =0, \\ &f_0(x,\mu_{\theta_0}(x),w_0) = x+a_0 x+b_0+w_0, \\ &c_1(x,\mu_{\theta_1}(x)) = \frac{1}{4}(a_1 x+b_1)^4 -\frac{1}{2}(a_1 x+b_1)^2 + x^2, \\ &f_1(x,\mu_{\theta_1}(x),w_1) = x+a_1 x+b_1+w_1, \\ &c_2(x,\mu_{\theta_2}(x)) = 0 \ \text{where} \ w_0,w_1 \overset{\mathrm{iid}}{\sim} \ \textit{Uniform}\biggr(-\sqrt{\frac{5}{3}}, \sqrt{\frac{5}{3}}\biggr), 
\end{align*}
where $\Theta=[-2, 2]\times[-2,2]$. The associated one-shot problem can be written as
\begin{align*}
\min_{-2 \leq b_0,a_1,b_1 \leq 2} \  \mathbb{E}_{w_0} \biggr[ &\frac{1}{4}\{a_1 (b_0+w_0)+b_1\}^4 \\ &-\frac{1}{2}\{a_1 (b_0+w_0)+b_1\}^2 + (b_0+w_0)^2 \biggr] 
\end{align*}

It turns out that there are 9 stationary points of the one-shot optimization in the interior of $\Theta$: $(b_0, a_1, b_1)=(0,\pm 0.7071, \pm 0.4082), (0, \pm 1, 0), (0, 0, \pm 1), (0,0,0)$. Among them, there are 4 strict local minimizers of the one-shot optimization: $(0, \pm 1, 0)$, $(0, 0, \pm 1)$. On the other hand, considering $\nabla_{\theta_1} c_1(x,\mu_{\theta_1}(x)) = [g(x, a_1, b_1)x, g(x, a_1, b_1)]$ where $g(x, a_1, b_1)=(a_1 x+b_1)(a_1 x+b_1 - 1)(a_1 x+b_1 + 1)$, there are 3 stationary points of DP: $(0, 0, \pm 1), (0,0,0)$ and 2 strict local minimizers of DP: $(0, 0, \pm 1)$. This verifies that a local minimizer (stationary point) of DP is indeed a local minimizer (stationary point) of the one-shot optimization but not the other way around. Fig.~\ref{fig:stochastic Landscape} shows the landscape of the one-shot optimization when $b_0$ is fixed to 0. 

% \begin{remark}
%     Fig.~\ref{fig:deterministic 3D} and Fig.~\ref{fig:stochastic Landscape} show that the optimization landscape of the one-shot problem is more complex than that of DP under parameterized policy for both the deterministic and stochastic cases, respectively. In other words, if the one-shot problem has a \textit{low} complexity, so does the DP problem.
% \end{remark}

\begin{figure}[t]
    \centering
    \includegraphics[width=43mm]{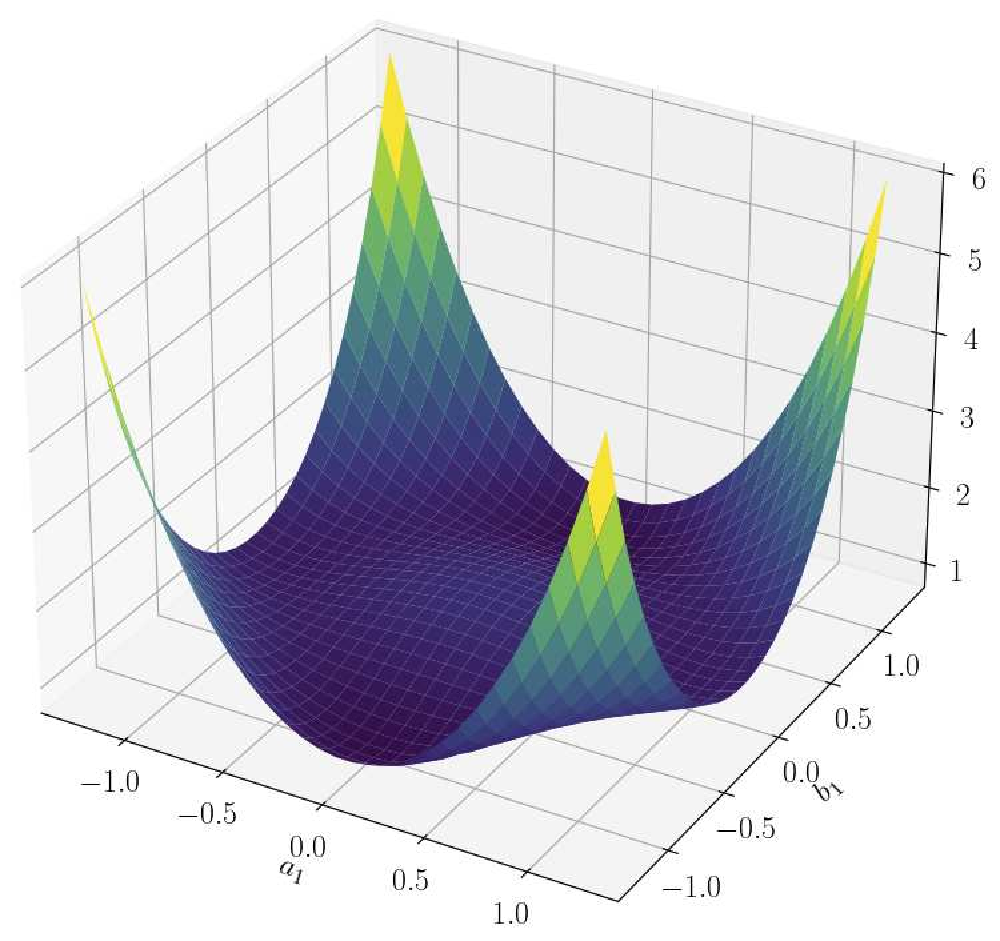}
    \caption{Landscape of the one-shot optimization for a stochastic parameterized problem: $b_0$ is fixed to $0$ in the figure. $(a_1,b_1)=(\pm 1, 0)$, $(0, \pm 1)$ are strict local minimizers of the one-shot optimization but only $(0, \pm 1)$ is a local minimizer of DP.}
    \label{fig:stochastic Landscape}
\end{figure}

Now, we present the specific case that a local minimizer of the one-shot optimization implies a local minimizer of DP, similar to Theorem \ref{thm: detparam one-shot to DP}. The preconditions of theorems are similar in the sense that they both consider the case when DP has a very low complexity in the sense that there is no spurious local minima at each step of DP. The main difference between the theorems comes from whether we consider a single trajectory or the expectation over infinitely many trajectories. We consider this in the view of stationarity. (see Definitions \ref{def: stationary control policy of DP}, \ref{def: detparam stationary point of the one-shot optimization}, and \ref{def: detparam stationary point of DP}) 
 
\begin{assumption}\label{assumption single stationary}
There exists only a single stationary control policy $\phi=(\phi_0,\dots,\phi_{n-1})$ which is also a locally minimum control policy in the interior of $A$ for all $x \in \mathbb{R}^N$. The parameterized policy defined by Definition \ref{def: linear combination} contains $\phi$, with the associated parameters denoted by $\phi' = (\theta_0', \dots, \theta_{n-1}')$.
\end{assumption}

% \begin{assumption}\label{assumption Theta large}
%     $\Theta$ is large enough to contain every parameter $\theta'\in \mathbb{R}^m$ such that $\nabla_{\mu} Q_k^{\phi'}(x,\mu_{\theta'}(x))$ is $0$ for all $x\in\mathbb{R}^N$ for any $k\in\{0,\dots,n-1\}$, where $\phi'$ is defined in Assumption \ref{assumption single stationary}. 
% \end{assumption}

\begin{theorem}\label{thm: stoparam one-shot to DP}
% Assume that Assumptions \ref{assumption single stationary} and \ref{assumption Theta large} hold.
Assume that Assumption \ref{assumption single stationary} holds.
Consider a local minimizer of the one-shot optimization $\pi=(\theta_0^*, \dots, \theta_{n-1}^*)$ in the interior of $\Theta^n$. If $x_k^*$ is a continuous random variable for all $k\in\{0,\dots,n-1\}$, where $(x_0^*, \dots,x_n^*)$ is the random state process associated with $\pi$, then $\pi$ is a local minimizer of DP.
\end{theorem}

\begin{proof}
Since $\phi$ is a single locally minimum control policy, $Q_k^{\phi'}(x,u)$ has no spurious local minima for all $k\in \{0,\dots,n-1\}$. Thus, by Proposition \ref{prop: action and parameter}, the corresponding $\inf_{x\in\mathbb{R}^N} \epsilon_k^* (x)=\infty > 0$ makes $\phi'$ be a local minimizer of DP. 
% Now, by Theorem \ref{thm: stoparam loc min}, $\phi'$ is a local minimizer of the one-shot optimization.

Consider a stationary point of the one-shot optimization $\pi=(\theta_0^*, \dots, \theta_{n-1}^*)$ in the interior of $\Theta^n$.
We will now prove by a backward induction that $\pi$ should always be $\phi'$; \textit{i.e.}, $\theta_{k}^*=\theta_{k}'$ for all $k\in\{0,\dots,n-1\}$. 

For the base step, at step $n-1$, since the parameterized policy contains $\phi_{n-1}$, it can be expressed as $\phi_{n-1} (x) = \sum_{i=1}^{m} s_i f_i(x)$, where $f_i : \mathbb{R}^N \rightarrow \mathbb{R}^M$, $i=1,\dots,m$, $f_i(x)=[f_{i1}(x), \dots, f_{iM}(x)]^T$ and $\theta_{n-1}'=(s_1,\dots,s_m)\in \Theta$. Notice that $ Q_{n-1}^{\phi'} (x,\mu_{\theta_{n-1}'}(x))= Q_{n-1}^{\pi} (x,\mu_{\theta_{n-1}'}(x))$ since $\theta_{n-1}'$ is the final parameter of the whole system to determine the control inputs and the state transition. Now, observe that
\[
\nabla_{\mu}Q_{n-1}^{\phi'} (x,\mu_{\theta_{n-1}'}(x)) = \nabla_{\mu}Q_{n-1}^{\pi} (x,\mu_{\theta_{n-1}'}(x)) = 0
\]
and $\mu_{\theta_{n-1}'}(x)$ is the unique solution for $\nabla_{\mu}Q_{n-1}^{\phi'} (x,\cdot) = 0$ since $\mu_{\theta_{n-1}'}(x)$ is the unique stationary point located within the interior of $A$ due to Assumption \ref{assumption single stationary}.
This yields the following expression with $\mu_{\theta}(x)=(u_1,\dots,u_M)^T$:
\begin{align*}
\nabla_{\mu}Q_{n-1}^{\phi'} (x,\mu_{\theta}(x)) &= \nabla_{\mu}Q_{n-1}^{\pi} (x,\mu_{\theta}(x)) \\ &=
\begin{bmatrix}
    (u_1 - \sum_{i=1}^m s_i f_{i1}(x))\cdot g_1 (x,\theta) \\
    \vdots \\
    (u_M - \sum_{i=1}^m s_i f_{iM}(x))\cdot g_M (x,\theta)
\end{bmatrix},
\end{align*}
where $g_j(x,\theta), j=1,\dots,M$, are nonnegative at $\theta = \theta_{n-1}'$ and positive at all the other points since $\phi'$ is a local minimizer of DP that yields the unique stationary control policy $\phi$.

Now, let $\mu_{\theta_{n-1}^*}(x)$ be $\sum_{i=1}^{m} d_i f_i(x)$ where $\theta_{n-1}^* = (d_1,\dots,d_m)$. According to the chain rule, it holds that 
\begin{align}\label{chain rule}
\nonumber &\nabla_{\theta_{n-1}}Q_{n-1}^{\pi}(x, \mu_{\theta_{n-1}^*}(x))^T \\  &= \nabla_{\mu}Q_{n-1}^{\pi}(x, \mu_{\theta_{n-1}^*}(x))^T \mathbf{D}^{\mu}_x(\theta_{n-1}^*) \\ \nonumber &= 
\begin{bmatrix}
    (\sum_{i=1}^m (d_i-s_i) f_{i1}(x))\cdot g_1 (x,\theta_{n-1}^*) \\
    \vdots \\
    (\sum_{i=1}^m (d_i-s_i) f_{iM}(x))\cdot g_M (x,\theta_{n-1}^*)
\end{bmatrix}^T 
\begin{bmatrix}
    f_{11}(x)  \cdots  f_{m1}(x) \\
    \vdots \quad \quad \ddots \quad \quad \vdots \\
    f_{1M}(x) \cdots f_{mM}(x)
\end{bmatrix}
\end{align}
Note that $\nabla_{\theta_{n-1}}\mathbb{E}_{x_0}[J_0^\pi (x_0)]=0$ since $\pi$ is a stationary point of the one-shot optimization in the interior of $\Theta^n$. It follows that
\begin{align}\label{expectation zero}
\nonumber &\nabla_{\theta_{n-1}}\mathbb{E}_{x_0}[J_0^\pi (x_0)] \\  &= \nabla_{\theta_{n-1}}\mathbb{E}_{x_0, w_0, \dots, w_{n-2}}[Q_{n-1}^\pi (x_{n-1}^*, \mu_{\theta_{n-1}^*}(x_{n-1}^*))] \\ \nonumber &= \mathbb{E}_{x_0, w_0, \dots, w_{n-2}}[\nabla_{\theta_{n-1}}Q_{n-1}^\pi (x_{n-1}^*, \mu_{\theta_{n-1}^*}(x_{n-1}^*))]=0.
\end{align}
The second equality comes from $Q_{n-1}^\pi (x,\mu_{\theta}(x))$ being differentiable with respect to the parameters due to the linearity of the policy defined by Definition \ref{def: linear combination}. Now, we substitute (\ref{chain rule}) into (\ref{expectation zero}) to derive an $m$-dimensional vector equation, and multiply $(d_k-s_k)$ with $k^{\text{th}}$ component as follows:
\vspace{-1mm}
\begin{align*}
\mathbb{E}_{x_0, w_0, \dots, w_{n-2}}\biggr[&\sum_{j=1}^M (d_k-s_k) f_{kj}(x_{n-1}^*)\cdot \\ &\sum_{i=1}^m (d_i-s_i) f_{ij}(x_{n-1}^*) \cdot g_j (x_{n-1}^*, \theta_{n-1}^*)\biggr]=0, 
\end{align*}
for all $k=1,\dots,m.$ We sum up the $m$ equations and rearrange the terms to derive the following equation:
\begin{align}\label{squared}
 \nonumber \sum_{j=1}^M \mathbb{E}_{x_0, w_0, \dots, w_{n-2}}\biggr[ \biggr(\sum_{i=1}^m (d_i-&s_i) f_{ij}(x_{n-1}^*)\biggr)^2 \cdot \\ &g_j (x_{n-1}^*, \theta_{n-1}^*)\biggr]=0.
\end{align}
% where the expectation is over $x_0, w_0, \dots, w_{n-2}$. 
The term inside the expectation is always nonnegative regardless of the distribution of $x_0, w_0, \dots, w_{n-2}$. Now, suppose that $\theta_{n-1}^*\neq \theta_{n-1}'$; \textit{i.e.}, $d_i \neq s_i$ for some $i \in \{1,\dots, m\}$. Then, we have $g_j (\cdot, \theta_{n-1}^*)$ to be strictly positive. As a result, for (\ref{squared}) to be satisfied, $\sum_{i=1}^m (d_i-s_i) f_{ij}(x_{n-1}^*)$ should be $0$ for every $j \in \{1,\dots,M\}$ for all possible values of $x_{n-1}$. Recall from Remark \ref{discrete continuous} to note that it is impossible to satisfy (\ref{squared}) since $x_{n-1}^*$ is a continuous random variable and the policy is defined by Definition \ref{def: linear combination}, \textit{i.e.}, a linear combination of some independent basis functions. Thus, $d_i = s_i$ for all $i \in \{1,\dots, m\}$, which means $\theta_{n-1}^*=\theta_{n-1}'$. 

For the induction step, assume that $\theta_{k}^*=\theta_{k}'$. Again, $ Q_{k}^{\phi'} (x,\mu_{\theta_{k}}(x))= Q_{k}^{\pi} (x,\mu_{\theta_{k}}(x))$ holds, and thus one can apply the same logic as the base step to obtain $\theta_{k-1}^*=\theta_{k-1}'$. 

Thus, $\pi=\phi'$ holds, which implies that $\pi$ is a local minimizer of DP since $\phi'$ is a local minimizer of DP.
\end{proof}

\begin{remark}
    The results of both Theorems \ref{thm: detparam one-shot to DP} and \ref{thm: stoparam one-shot to DP} state that a local minimizer of the one-shot optimization is indeed a local minimizer of DP under the common assumption that no spurious local minima exist at each step of DP. By taking the contrapositive, one can observe that under such a condition, there is at most one local minimizer of the one-shot optimization, indicating that no spurious local minima exist; \textit{i.e.}, if DP has a \textit{very low} complexity, the same holds for the one-shot problem.
\end{remark}

\begin{remark}
% Notice that the proof leverages the fact that $\pi$ is a stationary point of the one-shot optimization. Thus, Theorem \ref{thm: stoparam one-shot to DP} can also be stated for the stationary points instead of the local minimizers of the one-shot optimization. 
    To determine the form of $\nabla_{\mu}Q_{n-1}^{\phi'} (x,u)$, it was necessary to argue that $\mu_{\theta_{n-1}'}(x)$ should be the unique solution for $\nabla_{\mu}Q_{n-1}^{\phi'} (x,u) = 0$. For this to be true, there should certainly be only a single stationary control policy, which necessitates Assumption \ref{assumption single stationary}.
    In fact, it may be difficult to satisfy the precondition that a single stationary control policy should be in the interior of $A$ for all $x \in \mathbb{R}^N$. Instead, we can relax this condition to apply only within the domain of $x$; \textit{i.e.}, the set of values that at least one of the states $x_0, x_1, \dots, x_n$ can take. For example, if the state space is finite, satisfying the condition becomes relatively straightforward.
\end{remark}

\begin{remark}
    The challenging part of a backward induction in the proof arises from the fact that the state at step $k$ is fully determined by the previous steps but one cannot look at the previous steps in the backward induction. Thus, the main idea of the proof leverages equation (\ref{squared}), which incurs the fact that $\theta_k^* = \theta_k'$ regardless of the distribution of $x_0, w_0, \dots, w_{k-1}$. Thus, we only need the assumption that there is a single stationary control policy with respect to the given distribution of $x_0, w_0, \dots, w_{n-1}$. This is a big improvement from the work \cite{bhandari2024global} (see Condition 4 of Section 5.4) in the sense that Condition 4 needs no sub-optimal stationary point with respect to any possible distribution. 
\end{remark}
% \clearpage

    % This is possible only when $\Theta$ is large enough, which necessitates Assumption \ref{assumption Theta large}. Interestingly, Assumption \ref{det theta large} also says that the concept of $\Theta$ being large enough to cover the action space is crucial in the deterministic setting to be able to establish a relationship from one-shot to DP, \textcolor{blue}{which is actually the sufficient condition of Assumption \ref{assumption Theta large}. Thus, one can check Assumption \ref{assumption Theta large} in advance by alternatively examining Assumption \ref{det theta large}.}

\begin{figure}[t]
\centering {\includegraphics[width=45mm]{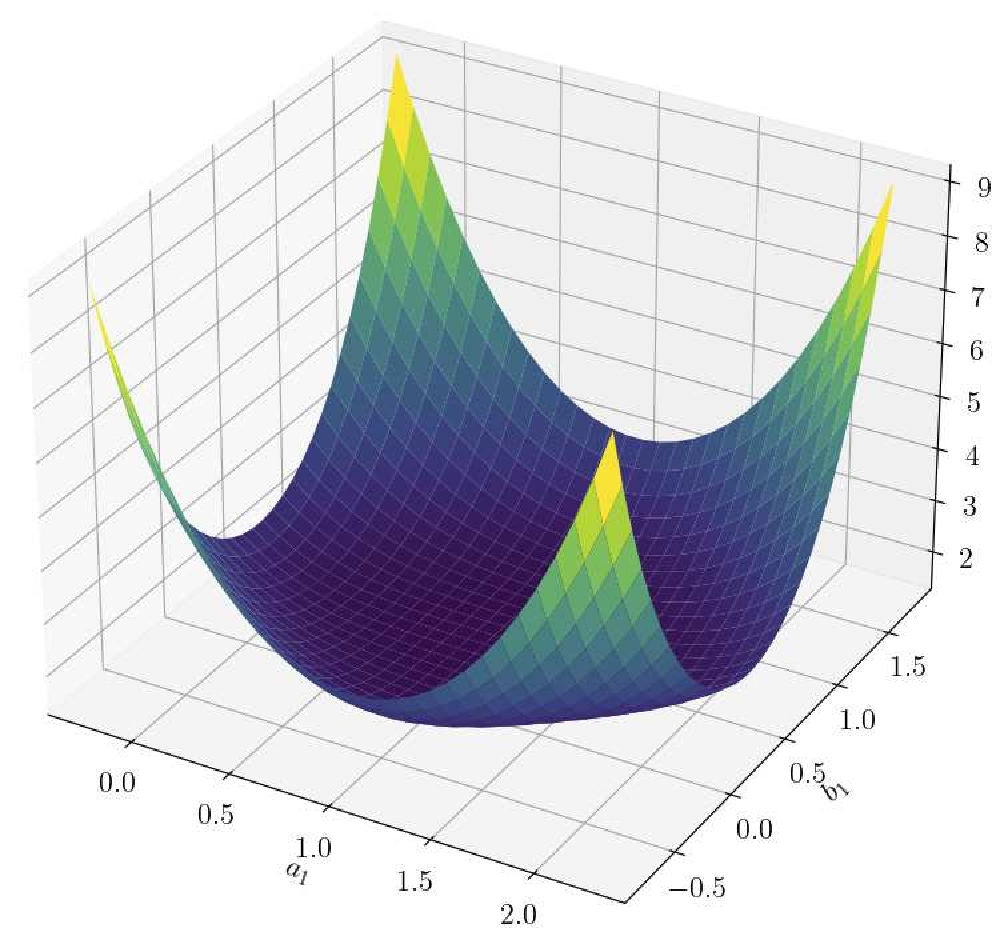}} \caption{Landscape of the one-shot optimization under the assumptions of Theorem \ref{thm: stoparam one-shot to DP}: $b_0$ is fixed to $0$ in the figure. $(a_1,b_1)=(1,0.5)$ is the only stationary point (local minimizer) of DP and also the only stationary point (local minimizer) of the one-shot optimization.}
% \vspace{-4mm}
\label{fig: equivalent stochastic Landscape}
\end{figure}

% \textcolor{blue}{Specifically, if $\Theta$ does not sufficiently cover the action space $A$, a stationary point of the one-shot optimization may not be a stationary point of DP. For example, one can truncate the parameter space of Fig.~\ref{fig:stochastic Landscape}
% % showing that a stationary point of the one-shot optimization is not necessarily a stationary point of DP.
% to $\Theta = [-0.5, 2] \times [0.1, 2]$ to have only a single stationary control policy $\mu_{\theta_0}(x_0)=0, \mu_{\theta_1}(x_1)=1$, which is a constant locally minimum control policy.
% % in the interior of an arbitrarily large action space since the policy is constant. 
% However, as shown in Fig.~\ref{fig: truncated stochastic Landscape}, there exist two stationary points of the one-shot optimization, $(0, 0, 1)$ and $(0,0.7071, 0.4082)$. The former is a stationary point and a local minimizer of DP, but the latter is not a stationary point of DP. Thus, this example necessitates large enough parameter space for Theorem \ref{thm: stoparam one-shot to DP}.}

Now, we present the pictorial example of Theorem \ref{thm: stoparam one-shot to DP}. Consider the example of 2-step problem presented above in Fig.~\ref{fig:stochastic Landscape}, but modify $c_1(x,\mu_{\theta_1}(x))$ to $\frac{1}{4}(a_1 x+b_1-x-0.5)^4 + x^4$. The associated one-shot problem can be written as 
\[
\min_{-2 \leq b_0,a_1,b_1 \leq 2} \  \mathbb{E}_{w_0} \biggr[ \frac{1}{4}\{(a_1-1) (b_0+w_0)+b_1 - 0.5\}^4 + (b_0+w_0)^4 \biggr] 
\]

% Let the action space be $[-20,20]$. The action space is large enough to have the locally minimum control policy to be in the interior of $A$, 
$(\pi_0, \pi_1)= (0, x+0.5)$ is the only locally minimum control policy, and the parameterized policy class contains this policy as $(b_0, a_1, b_1)=(0,1,0.5)$. Clearly, it is a local minimizer of DP. It turns out that the corresponding one-shot problem also has a single stationary point $(0,1,0.5)$, which is also a local minimizer of the one-shot optimization. Fig.~\ref{fig: equivalent stochastic Landscape} shows the landscape of the one-shot optimization when $b_0$ is fixed to $0$.

Considering both Theorems \ref{thm: stoparam loc min} and $\ref{thm: stoparam one-shot to DP}$, one can conclude that under
the assumptions of Theorem $\ref{thm: stoparam one-shot to DP}$, a local minimizer of DP is \textit{equivalent} to a local minimizer of
the one-shot optimization. 
% In other words, if DP has a very low complexity, the same holds for the one-shot problem. 

\subsection{Numerical Experiments}

In this subsection, we will present a high-dimensional experiment on the classical linear quadratic regulator (LQR):
\begin{align*}
&f_k(x_k, u_k) = A_k x_k + B_k u_k, \quad k=0,\dots,n-1, \quad x_0 \sim \mathcal{D},\\ &c_k (x_k, u_k) = x_k^\top Q_k x_k + u_k^\top R_k u_k,  \quad k=0,\dots,n-1, \\& u_k = K_k x_k, \quad k=0,\dots,n-1, \quad 
c_n(x_n) = x_n^\top Q_n x_n,
\end{align*}
whose goal is to find the optimal parameters: $K_0, \dots, K_{n-1}$.
% which is well-known to be the optimal policy.
% derived by discrete-time algebraic Riccati equation (DARE). 
To solve the problem using DP, we use Xpress Optimizer v9.3.0 \cite{xpress}. To solve the problem in a one-shot fashion, we use Gurobi Optimizer v11.0.0 \cite{gurobi} with the tolerance of $10^{-4}$.

Let $K_{k, \text{DP}}^*$ and $K_{k, \text{OS}}^*$ denote an observed local solution of the $k^\text{th}$ step parameter $K_k$ obtained by DP and the one-shot problem, respectively. We aim to determine whether each local solution of DP corresponds to some local solutions of the one-shot problem, and vice versa. One can verify this by first solving DP or the one-shot problem and then providing its solution as the initial parameter values when solving its counterpart. This method is often referred to as ``warm start." We expect to observe unchanged values from the initial guess if there is indeed a correspondence. Let $K_{k, \text{DP}\to\text{OS}}^*$ ($K_{k, \text{OS}\to\text{DP}}^*$) denote the solution of $K_k$ obtained by the one-shot problem (DP) using warm start with DP (one-shot) solution as an initial guess.
The initial distribution $\mathcal{D}$ introduces the stochasticity to the system and induces the states to be continuous random variables, which obeys the assumption of Theorem \ref{thm: stoparam one-shot to DP}.

We perform 20 experiments for $x_k \in \mathbb{R}^3$ and $u_k \in \mathbb{R}^4$ with $n=30$.  We randomly generate $A_k$ and $B_k$, whose entries are all in $[-100, 100]$. We also generate $Q_k = QQ^T$ and $R_k = RR^T + 100I$, where all entries of $Q$ and $R$ are in $[-20, 20]$ and $I$ denotes the identity matrix.
% This follows the rule of $R_k \succ 0$ to apply DARE.
$\mathcal{D}$ is the normal distribution with the expectation $20\mathbbm{1}$ and the variance $VV^T$, where $\mathbbm{1}$ denotes the vector of ones and all entries in $V$ are in $[-200, 200]$.
We consider two scenarios: (a) unconstrained and (b) constrained by the last-step (nonconvex) condition $K_{n-1}^T K_{n-1} \succeq 10000I$, where $\succeq$ denotes the Loewner partial ordering (roughly speaking, this condition ensures that the controller has a high gain).
Table \ref{DP oneshot lqr} shows whether the correspondence holds between the solutions of DP and the one-shot problem using warm start under the two scenarios, presenting the results of the average of 20 experiments.

It turns out that for both scenarios, one can observe that a solution of DP corresponds to each solution of one-shot problem since the one-shot solver directly identifies DP solution as a local solution of the one-shot problem without any numerical update. This implication supports the findings of Theorem \ref{thm: stoparam loc min}.

However, whether a one-shot solution implies some DP solutions depends on the problem setting. Our experiment for the unconstrained case shows that a solution of the one-shot problem indeed corresponds to that of DP, implying with the above result that a one-shot solution is equivalent to a DP solution. 
Previous studies have shown the global convergence of this one-shot problem by proving that LQR satisfies the gradient dominance property, even though the problem is generally nonconvex \cite{fazel2018lqr, hambly2021lqr}. Our general approach alternatively observes the DP counterparts: Since every DP sub-problem of LQR has no spurious local minima, our experiment implies that the one-shot LQR problem also has none of them and achieves the global convergence, which supports Theorem \ref{thm: stoparam one-shot to DP}.

On the other hand, the nonconvex constraint on $K_{n-1}$ creates spurious solutions for the $(n-1)^\text{th}$ DP, independent of whether the $(n-2)^\text{th}, \dots, 0^\text{th}$ DP steps have any spurious local minima. 
Our experiment for the constrained case shows that having spurious local minima at the $(n-1)^\text{th}$ DP propagates backward to $K_1$, where the solver fails to guarantee that an observed local minimum of the one-shot optimization corresponds to that of DP. This illustrates that the landscape of the one-shot problem has a higher complexity than its DP counterpart, which was also shown in Fig.~\ref{fig:deterministic 3D} and Fig.~\ref{fig:stochastic Landscape}. This result serves as a counterexample of the converse of Theorems \ref{thm: stoparam loc min} and \ref{thm: stoparam stationary}, and it implies that a single high-complexity DP step affects the landscape of the one-shot problem.
% Thus, if the one-shot problem has low complexity, each step of DP can easily be solved.

\begin{table}[t]
\begin{threeparttable}
\scriptsize
\centering
\caption{Relationship between DP and one-shot solutions of LQR}
\begin{tabular}{ |c|c|c| } 
 \hline
 \multirow{2}{*}{\backslashbox[45mm]{Numerical difference}{Scenario}} & \multirow{2}{*}{(a) Unconstrained}  & \multirow{2}{*}{(b) Constrained}  \\ && \\\hline\hline
 $\|K_{0,\text{DP}}^*-K_{0,\text{DP}\to\text{OS}}^*\|_F/\|K_{0,\text{DP}}^*\|_F$ & 0 & 0\\ 
 \hline
 $\|K_{1,\text{DP}}^*-K_{1,\text{DP}\to\text{OS}}^*\|_F/\|K_{1,\text{DP}}^*\|_F$ & 0 & 0 \\ \hline\hline
 $\|K_{0,\text{OS}}^*-K_{0,\text{OS}\to\text{DP}}^*\|_F/\|K_{0,\text{OS}}^*\|_F$ & $2.901\cdot10^{-6}$ & $5.399\cdot 10^{-4}$ \\ \hline
 $\|K_{1,\text{OS}}^*-K_{1,\text{OS}\to\text{DP}}^*\|_F/\|K_{1,\text{OS}}^*\|_F$ & $2.291\cdot10^{-5}$ & $\textbf{3.156}$ \\ \hline
 % \hline \multicolumn{3}{|c|}{The ratio of $\|K_{0,\text{DP}}^*\|_F$ between two scenarios $=0.9999$}  \\ 
\end{tabular}
% \vspace{-1mm}
\label{DP oneshot lqr}
\begin{tablenotes}
    \scriptsize
    \item The bold value (3.156) being far from 0 supports that the landscape of the one-shot problem is more complex than its DP counterpart. One can find out that the other values in Table II are 0 or close to 0.
\end{tablenotes}
\end{threeparttable}
\end{table}

\section{Conclusion}
\label{sec:con}
In this paper, we studied the optimization landscape of the optimal control problems via two different formulations: one-shot optimization aimed at solving for all input values at the same time, and DP method aimed at finding the input values sequentially. 
% We introduced the notions of spurious (non-global) local minimizers for the one-shot problem and spurious locally minimum control policies for DP. 
For the deterministic problem, we proved that under some mild conditions, each local minimizer of the one-shot optimization corresponds to an input sequence induced by some locally minimum control policy of DP, and vice versa. 
% We also proved that the control sequence induced by a stationary control policy of DP corresponds to some stationary point of the one-shot problem.

To help better understand the quality of the local solutions obtained by reinforcement learning algorithms, we incorporated exact parameterized policies into the optimal control problem for both deterministic and stochastic dynamics. 
We showed that if the one-shot problem has a low complexity, so do the corresponding DP sub-problems, indicating the success of DP methods.
% Fig.~\ref{fig:deterministic 3D} and Fig.~\ref{fig:stochastic Landscape} show that under a parameterized policy, the optimization landscape of the one-shot problem is more complex than its DP counterpart for both the deterministic and stochastic cases, respectively. 
% In other words, if the one-shot problem has a \textit{low} complexity, so does the DP problem.
% We showed that each local minimizer (stationary point) of DP corresponds to a local minimizer (stationary point) of the one-shot optimization, but not the other way around. These discoveries show that the optimization landscape of the one-shot optimization is more complex than its DP counterpart, and thus the one-shot problem having a low complexity implies a low complexity for DP. 
Moreover, under the condition that there exists only a single locally minimum control policy, with different technical assumptions, both deterministic and stochastic cases yield that a local minimizer of the one-shot optimization is equivalent to a local minimizer of DP. 
% This implies that DP having a \textit{very low} complexity implies a very low complexity for the one-shot problem, and furthermore, a local minimizer of both methods is equivalent.

We focused on the discrete-time finite-horizon optimal control problem in this work. A natural future direction would be to extend this work to the continuous-time and infinite-horizon cases, which was discussed in Remark \ref{allresult}.
For safety-critical systems, state constraints may also be enforced, with recursive feasibility being crucial to guarantee the success of DP.
% On the other hand, a limitation of this work is our restriction of the policy class. A natural extension would be to 
One may also want to extend the parameterized policy class beyond a linear combination of basis functions, such as composite functions widely used in deep neural networks.

% \raggedbottom

% \section*{References}
% \vspace{-3.2mm}
% \bibliographystyle{IEEEtran}
% \bibliography{references}

\printbibliography

% \vspace{-2mm}

\begin{IEEEbiography}[{\includegraphics[width=1in,height=1.25in,clip,keepaspectratio]{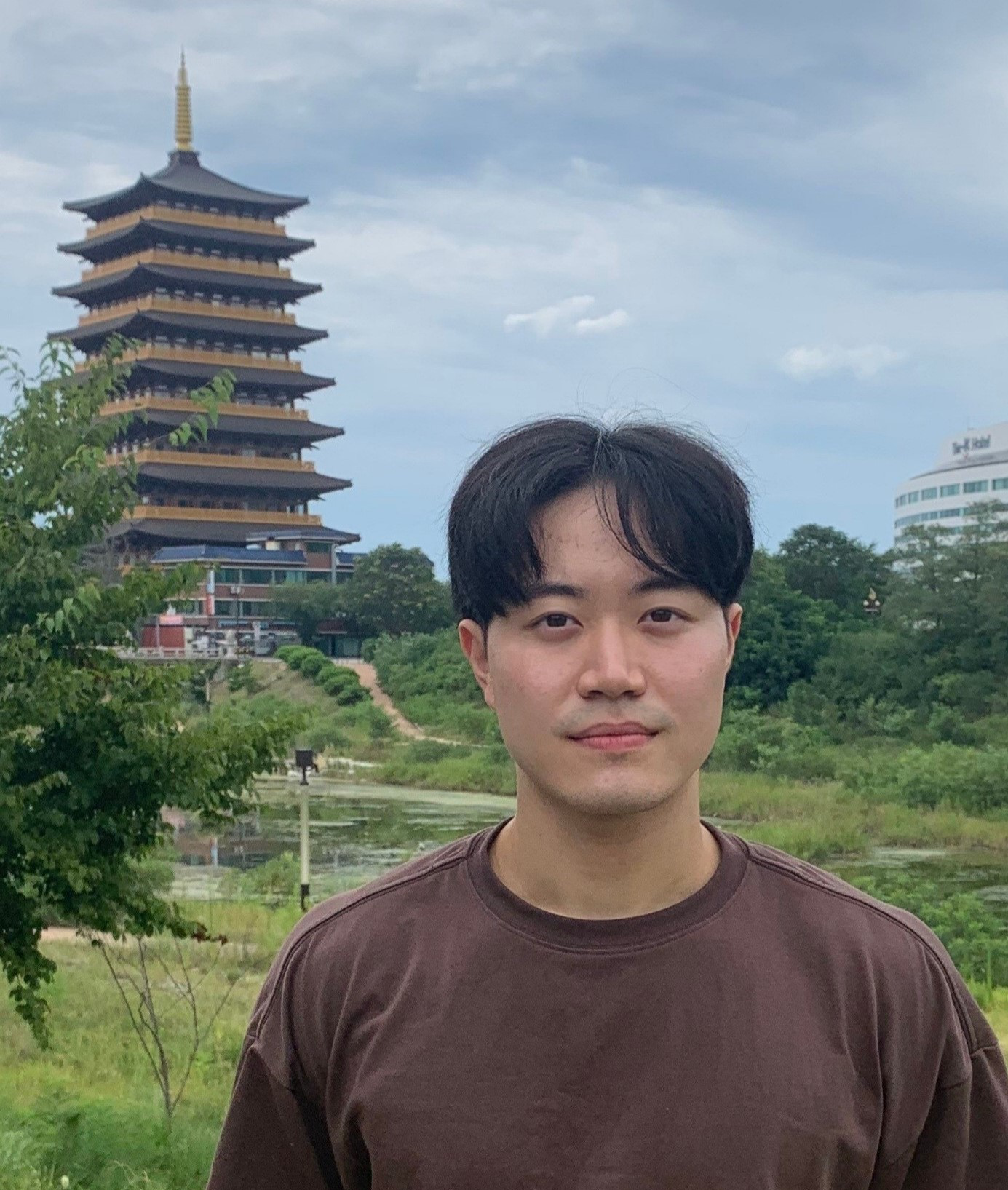}}]{Jihun Kim} is a Ph.D. candidate in Industrial Engineering and Operations Research at the University of California, Berkeley, CA, USA. He obtained the B.S. degree in both Industrial Engineering and Statistics from Seoul National University in 2022. His research interests include nonconvex optimization, dynamical systems, and learning-based control, along with their applications to safety-critical systems such as robotics, aircraft, and autonomous driving.
\end{IEEEbiography}

% \vspace{-2mm}

\begin{IEEEbiography}[{\includegraphics[width=1in,height=1.25in,clip,keepaspectratio]{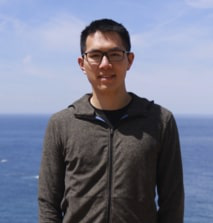}}]{Yuhao Ding} received the Ph.D. degree in Industrial Engineering and Operations Research at the University of California, Berkeley, CA, USA. He obtained the B.E. degree in Aerospace Engineering from Nanjing University of Aeronautics and Astronautics in 2016, and the M.S. degree in Electrical and Computer Engineering from University of Michigan, Ann Arbor in 2018.
\end{IEEEbiography}

% \vspace{-2mm}

\begin{IEEEbiography}[{\includegraphics[width=1in,height=1.25in,clip,keepaspectratio]{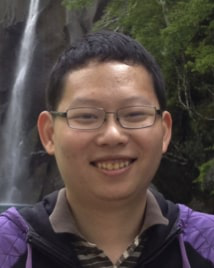}}]{Yingjie Bi} received the B.S. degree in Microelectronics from Peking University in 2014 and the Ph.D. degree in Electrical and Computer Engineering from Cornell University in 2020. He is currently a postdoctoral scholar in the department of Industrial Engineering and Operations Research at the University of California, Berkeley. His research interests include nonconvex optimization, machine learning, computer networks and control theory.
\end{IEEEbiography}

% \vspace{-2mm}

\begin{IEEEbiography}[{\includegraphics[width=1in,height=1.25in,clip,keepaspectratio]{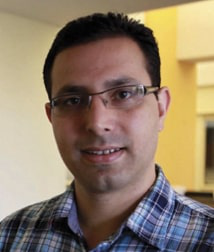}}]{Javad Lavaei} (Fellow, IEEE) is an Associate Professor in the Department of Industrial Engineering and Operations Research at UC Berkeley. He obtained the Ph.D. degree in Control \& Dynamical Systems from California Institute of Technology. He is a senior editor of the IEEE Systems Journal and has served on the editorial boards of the IEEE Transactions on Automatic Control, IEEE Transactions on Control of Network Systems, IEEE Transactions on Smart Grid, and IEEE Control Systems Letters. 
\end{IEEEbiography}

\end{document}